\documentclass[reqno]{amsart}
 
\usepackage[english]{babel}
\usepackage[latin1]{inputenc}

\usepackage{yhmath}

\usepackage{mathabx}

\usepackage{amssymb}
\usepackage{amscd,amsmath,amsfonts}
\usepackage{latexsym}

\usepackage{calc}
\usepackage{verbatim}
\usepackage{url}
\usepackage[colorlinks]{hyperref}
\usepackage{stackrel}
\usepackage{cancel}
\usepackage{enumitem}
\usepackage{bbm}

\usepackage{color}
\usepackage{eepic}
\usepackage{picture}
\usepackage{rotating}
\usepackage{graphics,graphicx}

\usepackage{epstopdf}
\usepackage{tikz}

\newtheorem{thm}{Theorem}[section]
\newtheorem{prop}[thm]{Proposition}
\newtheorem{lem}[thm]{Lemma}
\newtheorem{cor}[thm]{Corollary}

\theoremstyle{definition}

\theoremstyle{remark}
\newtheorem{rem}[thm]{Remark}

\newcommand{\ie}[0]{i.e.\ }

\newcommand{\abs}[1]{\left\lvert#1\right\rvert}

\newcommand{\ceil}[1]{\left\lceil #1 \right\rceil}

\newcommand{\zz}{{\mathbb Z}}

\newcommand{\nz}{{\mathbb N}}
\newcommand{\rz}{{\mathbb R}}

\newcommand{\rd}{{\rz^d}}
\newcommand{\zt}{{\zz^2}}
\newcommand{\zd}{{\zz^d}}

\newcommand{\ig}{\ensuremath{\mathfrak{i}} }
\newcommand{\tg}{\ensuremath{\mathfrak{t}} }
\newcommand{\ag}{\ensuremath{{\mathcal A}} }

\newcommand{\fg}{\ensuremath{{\mathcal F}} }

\newcommand{\sg}{\ensuremath{{\mathcal S}} }

\newcommand{\heis}{{{\mathsf H}_3(\zz)}}
\newcommand{\xf}{\ensuremath{{\mathsf x}} }
\newcommand{\yf}{\ensuremath{{\mathsf y}} }
\newcommand{\zf}{\ensuremath{{\mathsf z}} }

\newcommand{\mybf}[1]{%
    \pdfliteral direct {2 Tr 0.73 w} 
     #1%
    \pdfliteral direct {0 Tr 0 w}%
}

\newcommand{\bfz}{\mybf{0}} 
\newcommand{\bfo}{\mybf{1}}

\DeclareRobustCommand\tilb{\ensuremath{{
\tikz[baseline=0.6mm]{
\draw[line width=0.5pt, color=gray] (0,0) -- (0.34,0) -- (0.34,0.34) -- (0,0.34) -- cycle;}}} }
\DeclareRobustCommand\tilf{\ensuremath{{
\tikz[baseline=0.6mm]{
\draw[line width=0.8pt] (0.17,0) -- (0.17,0.34);
\draw[line width=0.5pt, color=gray] (0,0) -- (0.34,0) -- (0.34,0.34) -- (0,0.34) -- cycle;}}} }
\DeclareRobustCommand\tild{\ensuremath{{
\tikz[baseline=0.6mm]{
\draw[line width=0.8pt] (0,0.34) -- (0.34,0);
\draw[line width=0.5pt, color=gray] (0,0) -- (0.34,0) -- (0.34,0.34) -- (0,0.34) -- cycle;}}} }
\DeclareRobustCommand\tilc{\ensuremath{{
\tikz[baseline=0.6mm]{
\draw[line width=0.8pt] (0.17,0) -- (0.17,0.34);
\draw[line width=0.8pt] (0,0.34) -- (0.34,0);
\draw[line width=0.5pt, color=gray] (0,0) -- (0.34,0) -- (0.34,0.34) -- (0,0.34) -- cycle;}}} }
\DeclareRobustCommand\tila{\ensuremath{{
\tikz[baseline=0.6mm]{
\draw[line width=0.8pt] (0.17,0) -- (0.17,0.34);
\draw[line width=0.8pt] (0,0.17) -- (0.34,0.17);
\draw[line width=0.8pt] (0,0.34) -- (0.34,0);
\draw[line width=0.5pt, color=gray] (0,0) -- (0.34,0) -- (0.34,0.34) -- (0,0.34) -- cycle;}}} }

\newcommand{\arcol}[0]{red }

\makeatletter
\definecolor{lightgray}{rgb}{.8,.8,.8}
\definecolor{gray}{rgb}{.6,.6,.6}
\definecolor{darkgray}{rgb}{.4,.4,.4}
\definecolor{lightred}{rgb}{1,.7,.7}
\definecolor{cmm}{RGB}{237,49,35}
\definecolor{green}{RGB}{0,150,0}
\newcommand{\const}{\mathop{\operator@font const}\nolimits}
\newcommand{\card}{\mathop{\operator@font card}\nolimits}
\newcommand{\diam}{\mathop{\operator@font diam}\nolimits}
\newcommand{\id}{\mathop{\operator@font Id}\nolimits}
\newcommand{\perio}{\mathop{\operator@font per}\nolimits}
\newcommand{\per}{\mathop{\operator@font Per}\nolimits}
\newcommand{\orb}{\mathop{\operator@font Orb}\nolimits}
\newcommand{\fix}{\mathop{\operator@font Fix}\nolimits}
\newcommand{\erz}{\mathop{\operator@font span}\nolimits}
\newcommand{\supp}{\mathop{\operator@font supp}\nolimits}
\newcommand{\stab}{\mathop{\operator@font Stab}\nolimits}
\newcommand{\lcm}{\mathop{\operator@font lcm}\nolimits}
\newcommand{\htop}{\mathop{\operator@font h_{\textnormal{top}}}\nolimits}
\newcommand{\hmu}{\mathop{\operator@font h}\nolimits}
\newcommand{\cf}{\mathop{\operator@font CF}\nolimits}
\newcommand{\GL}{\mathop{\operator@font GL}\nolimits}
\newcommand{\homeo}{\mathop{\operator@font Homeo}\nolimits}
\newcommand{\agR}{\mathop{{\mathcal A}_{\operator@font Rob}}\nolimits}
\newcommand{\agB}{\mathop{{\mathcal A}_{\operator@font Bin}}\nolimits}
\newcommand{\agD}{\mathop{{\mathcal A}_{\operator@font Seg}}\nolimits}
\newcommand{\agC}{\mathop{{\mathcal A}_{\operator@font Count}}\nolimits}
\newcommand{\agDtil}{\mathop{\widetilde{\mathcal A}_{\operator@font Seg}}\nolimits}
\newcommand{\agCtil}{\mathop{\widetilde{\mathcal A}_{\operator@font Count}}\nolimits}
\newcommand{\omR}{\mathop{{\Omega}_{\operator@font Rob}}\nolimits}
\newcommand{\omT}{\mathop{\widetilde{\Omega}}\nolimits}
\newcommand{\omM}{\mathop{{\Omega}_{\operator@font Min}}\nolimits}
\makeatother

\title[Strongly aperiodic SFT for the discrete Heisenberg group]{A strongly aperiodic shift of finite type on the discrete Heisenberg group using Robinson tilings}
\author{Ay\c se A. \c Sah\.in}
\address{Department of Mathematics and Statistics, Wright State University, Dayton, OH 45435, USA}
\email{ayse.sahin@wright.edu}
\author{Michael Schraudner}
\address{Centro de Modelamiento Matem\'atico, Universidad de Chile, 8370459 Santiago, Chile}
\email{mschraudner@dim.uchile.cl}
\author{Ilie Ugarcovici}
\address{Department of Mathematical Sciences, DePaul University, Chicago, IL 60614, USA}
\email{iugarcov@depaul.edu}
\thanks{The first author was partially supported by an Association for Women in Mathematics Travel Grant. The second author was supported by FONDECYT Project 1140015 and CMM ANID PIA Basal Grant AFB170001. The third author was partially supported by a Simons Foundation Collaboration grant.}

\begin{document}

\begin{abstract}
We explicitly construct a strongly aperiodic subshift of finite type for the discrete Heisenberg group. Our example builds on the classical aperiodic tilings of the plane due to Raphael Robinson. Extending those tilings to the Heisenberg group by exploiting the group's structure and posing additional local rules to prune out remaining periodic behavior we maintain a rich projective subdynamics on $\zt$ cosets. In addition the obtained subshift factors onto a strongly aperiodic, minimal sofic shift via a map that is invertible on a dense set of configurations.
\end{abstract}

\maketitle

\section{Introduction}\label{s:Intro}

One of the fundamental differences between $\zd$ symbolic dynamics for $d=1$ and $d>1$ is the existence of strongly aperiodic shifts of finite type (SFTs) in two and more dimensions. This leads to the undecidability of some very basic questions for higher dimensional systems, like the emptiness or the extension problem, and explains the strong influence of computability on the development of the subject. Since the 1970s there has been an increasing interest in extending the theory of symbolic dynamics to more general group actions. In this broader context it is natural to ask which finitely generated groups admit strongly aperiodic shifts of finite type, and what topological dynamical properties such examples might exhibit.

Here we consider symbolic actions of the discrete Heisenberg group given by the set of $3\times 3$ upper-triangular unipotent integer matrices
\[
\heis=\left\{{\footnotesize\begin{pmatrix} 1 & x & z\\0 & 1 & y\\ 0 & 0 & 1\end{pmatrix}}\,:\,x,y,z\in\zz\right\}
\]
equipped with the usual matrix multiplication operation. This non-abelian nilpotent linear group can also be seen as a semi-direct product $\zt\rtimes_A\zz$, where $A=\left(\begin{smallmatrix}
1& 0\\ 1 & 1 \end{smallmatrix}\right)$ induces an automorphism of $\zt$ leading to the group operation given by
\[
{\bigl(x,(y,z)\bigr)\cdot\bigl(a,(b,c)\bigr)=\bigl(x+a,\bigl((y,z)+A^x(b,c)\bigr)\bigr)=\bigl(x+a,(y+b,z+c+xb)\bigr)\ .}
\]
The authors announced in 2014 the first construction of a strongly aperiodic subshift of finite type on $\heis$ \cite{SSU}. That example was never published but motivated a number of related results in subsequent years which have led to the current state of the art regarding strongly aperiodic shifts of finite type. It is presented in Section~\ref{s:construction} and is the first step towards establishing the following theorem, the main result in this paper.

\begin{thm}\label{t:main}
The discrete Heisenberg group $\heis$ admits a strongly aperiodic shift of finite type which has a strongly aperiodic and minimal subshift factor via a map that is  at most two-to-one, and one-to-one on a dense set of configurations.
\end{thm}

We provide a summary of results that followed our announcement later in the paper but mention here that Barbieri and Sablik have shown the existence of strongly aperiodic shifts of finite type on semi-direct product groups of the form $\zt\rtimes H$, with $H$ a group with decidable word problem \cite{BS16}. Their proof involves a non-trivial extension of Hochman's intricate work in \cite{H09} and subsumes the case of the discrete Heisenberg group. Our construction is similar to theirs in that both extend $\zt$ tilings; their starting point are substitution shifts and ours are the Robinson tilings. Our results differ from theirs in that we exploit the group geometric properties of the discrete Heisenberg group to provide an explicit description of the (small) alphabet and local rules of the SFT, obtain both rich projective subdynamics of the SFT and obtain the minimal factor described in Theorem~\ref{t:main}.

Our arguments rely on a subtle relationship between the combinatorial structure of the $\zt$ subshift given by the Robinson tilings together with the geometry of the Cayley graph of $\heis$. As we show in Section~\ref{s:construction} the Cayley graph of the discrete Heisenberg group can be viewed as a countable union of $\zt$ cosets (drawn horizontally) which are connected by slanted (vertical) edges. We prove that it is possible to fill every other $\zt$ coset by a valid Robinson tiling while forcing additional local rules along the connecting vertical edges in a consistent fashion that eliminates any kind of periodicity. The idea of our approach partly resembles work of Culik and Kari in \cite{CK95} who constructed a strongly aperiodic $\zz^3$-SFT by extending their own strongly aperiodic $\zt$ tilings to the direct product $\zz^3\simeq\zt\times\zz$ in a related way.

\subsection*{Background on aperiodic shifts of finite type}
The existence of a strongly aperiodic shift of finite type was first established in the context of studying tilings of the plane. $\zt$ shifts of finite type can be modeled by tiling spaces obtained from tiling $\rz^2$ by aligned unit squares with colored edges. The SFT's local rules can be translated into the condition to have colors match across tile borders. The existence of a strongly aperiodic set -- \ie a set which tiles $\rz^2$, but only by producing strongly aperiodic configurations -- of such edge-colored squares for $\rz^2$ is thus equivalent to the existence of a strongly aperiodic $\zt$-SFT. Wang's famous conjecture \cite{W61} dating from the 1960s claimed that whenever a collection of edge-colored unit square tiles could tile the entire plane, then those tiles would also allow for a periodic tiling, which in turn would imply decidability of the emptiness problem asking whether a given finite collection of tiles will indeed tile the plane. Berger \cite{B66} was the first to disprove this conjecture with a huge set of over 20,000 distinct tiles whereas subsequent authors have constructed examples using fewer and fewer tiles. In this paper we use extensively the well-known Robinson tilings constructed by Raphael Robinson in 1971 \cite{R71}.  In 1996 Kari and Culik gave a strongly aperiodic example with 13 square tiles \cite{K96,C96}.  More recently Jeandel and Rao \cite{JR17} established the minimal cardinality for strongly aperiodic edge-colored square tile sets to be 11. We note that Berger's construction of a strongly aperiodic set of square tiles led to a proof of the algorithmic undecidability of the emptiness problem for $\mathbb Z^2$. As a consequence there is also no general method to determine whether or not a given two dimensional shift of finite type is non-empty and whether or not a locally admissible pattern on a finite part of $\zt$ extends to an entire configuration. (See for example \cite{RR04} for a more in depth treatment of these ideas and corresponding results.)

Related work on tilings for general groups is an active field and, deferring formal definitions until Section~\ref{s:GSFT}, we provide here an incomplete list of results. Results mainly come in two flavors, considering either continuous or (finitely generated) discrete groups. In the 1990s Block and Weinberger studied weakly aperiodic tiling spaces (namely none of the tilings have a co-compact symmetry)
for certain families of Lie groups \cite{BW92}, whereas Mozes \cite{M97} constructed strongly aperiodic tile sets on symmetric spaces obtained from some particular classes of semi-simple Lie groups. Later Goodman-Strauss \cite{GS05,GS10} extensively studied strongly aperiodic tiling spaces on the hyperbolic plane and Margenstern \cite{M08} proved the undecidability of the emptiness problem in this context. Fusing techniques by Goodman-Strauss with properties of the planar Kari-Culik tilings, in 2013 Aubrun and Kari \cite{AK13} were able to produce a weakly aperiodic tile set (with local matching rules) for the class of solvable Baumslag-Solitar groups, showing at the same time the undecidability of the emptiness problem for SFTs on these two-generator, one-relator groups. More recently Esnay and Moutot showed that these examples are even strongly aperiodic and established more examples for certain Baumslag-Solitar groups \cite{EM}.

There is also a more recent body of work addressing the same questions for shifts of finite type on other types of finitely generated groups using tools from a variety of different areas. As mentioned above, the authors established, several years ago, the existence of a strongly aperiodic SFT of $\heis$ and there have been several advances in the area since that announcement. Most significant and relevant to our work is, of course, the result of Barbieri and Sablik \cite{BS16} described above. In addition, Carol and Penland proved that admitting a strongly aperiodic shift of finite type is an invariant of commensurability of groups \cite{CP15}. Meanwhile Cohen showed that the existence of such subshifts is also a quasi-isometry invariant of finitely presented torsion free groups and that strongly aperiodic shifts of finite type can not exist in groups with two or more ends \cite{C14}. In later work with Goodman-Strauss he established the existence of strongly aperiodic shifts of finite type on hyperbolic surface groups \cite{CG17} and in work with Goodman-Strauss and Rieck \cite{CGR17} showed that the word-hyberbolic groups which admit free SFTs are the ones with at most one end.  We also mention the work of Barbieri \cite{Barb} who established the existence of strongly aperiodic SFT for certain branch groups, including the Grigorchuk group.

\subsection*{Organization of the paper}
The paper is organized as follows. In Section~\ref{s:GSFT} we define subshifts and shifts of finite type on finitely generated groups, briefly discuss relevant properties like aperiodicity and minimality and recall the notion of Cayley graphs. Section~\ref{s:Robinson} describes the $\zt$-SFT given by the Robinson tilings and a few technical properties of this SFT essential to our construction. After recalling some basic facts about the discrete Heisenberg group $\heis$, in Section~\ref{s:construction} we present the principal construction of a specific strongly aperiodic SFT on this group exploiting the geometry of its Cayley graph as well as the rigidity of the Robinson SFT. Section~\ref{s:minimal} slightly modifies our previous construction to obtain a $\heis$ SFT with an almost one-to-one, aperiodic, sofic factor.

\section{Subshifts on finitely generated groups}\label{s:GSFT}

Assuming a basic familiarity with symbolic dynamics we use this section to briefly recall a few key notions and set some notation to be utilized in what follows. For additional background and more details we refer the reader to \cite{CC10} and \cite{DM95}.

Let $G$ be a finitely generated, (countably) infinite group and let $\ag$ be a finite (discrete) set, called the {\em alphabet}, whose elements are referred to as {\em symbols}. The cartesian product $\ag^G$ -- equipped with the prodiscrete topology -- is the compact metric space consisting of all functions $\omega:\,G\rightarrow\ag$, called {\em configurations}. We use the notation $\omega=(\omega_g)_{g\in G}\in\ag^G$ where $\omega_g:=\omega(g)$ denotes the symbol seen at a particular element $g\in G$. Similarly, for any finite subset $F\subsetneq G$, an element $p\in\ag^F$ which can be thought of as $p:\,F\rightarrow\ag$, \ie the restriction of a configuration to a finite support, is called a {\em pattern of shape $F$}. Again the notation $p=(p_g)_{g\in F}\in\ag^F$ with $p_g:=p(g)$ is used, while $\supp(p)=F$ denotes its support. We call $p$ a {\em subpattern} of another pattern $q\in\ag^{F'}$ of (finite) shape $F'\subsetneq G$, respectively of a configuration $\omega\in\ag^G$, denoted by $p\sqsubseteq q$, respectively $p\sqsubseteq \omega$ if there exists an element $g\in G$ such that $F\!\cdot\! g\subseteq F'$ and $p=q|_{F\cdot g}$, respectively $p=\omega|_{F\cdot g}$. The countable set of all patterns is defined as $\ag^*:=\bigcup_{F\subsetneq G\text{ finite}}\ag^F$.

Since $G$ acts on itself by translations we obtain a natural $G$-action $G\stackrel{\sigma}{\curvearrowright}\ag^G$ given by homeomorphisms $\bigl\{\sigma^g:=\sigma(g):\,\ag^G\rightarrow \ag^G\bigr\}_{g\in G}$ determined coordinate-wise as $\bigl(\sigma^g(\omega)\bigr)_h:=\omega_{hg}$ for all $g,h\in G$. This symbolic $G$-action is known as the {\em (right) shift} and the pair $(\ag^G,\sigma)$ as the {\em full $G$-shift} over the alphabet $\ag$.

A subset of $\ag^G$ preserved under the $G$-action $\sigma$ is said to be {\em shift invariant}. A {\em symbolic dynamical system} for the group $G$ and the alphabet $\ag$ is then given by any closed, thus compact, and shift invariant set $\Omega\subseteq\ag^G$ together with the restriction of the shift action of $G$ to $\Omega$. The pair $(\Omega,\sigma|_\Omega)$ is called a {\em $G$-shift space} or simply a {\em $G$-subshift}. For notational convenience we refer to a shift space merely as $\Omega$, or when we want to emphasize the acting group, as $(\Omega,G)$.

It is easy to establish an equivalent -- but more combinatorial -- definition of $G$-shift spaces. In fact any subshift $(\Omega,G)$ over $\ag$ is characterized by selecting an (at most countable) family of finite patterns $\fg\subseteq\ag^*$ such that $\omega\in\ag^G$ is an element of $(\Omega,G)$ if and only if $\sigma^g(\omega)|_{\supp(p)}\neq p$ for all $g\in G$ and $p\in\fg$. To emphasize the role of the chosen {\em family of forbidden patterns $\fg$} we introduce the notation $\Omega_\fg:=\bigl\{\omega\in\ag^G\,:\,\forall\,p\in\fg:\,p\nsqsubseteq\omega\bigr\}$. Moreover we call $(\Omega,G)$ a {\em $G$-shift of finite type} ($G$-SFT), if it is possible to obtain $\Omega=\Omega_\fg$ for some finite family $\fg$. If, in addition, for all forbidden patterns $p\in\fg$, $\supp(p)$ consists of a pair of group elements connected by a generator then $\Omega$ is called a {\em nearest neighbor} $G$-SFT. The image of a $G$-SFT under a (topological) factor map is called a {\em sofic $G$-shift}.  A topological factor map is called {\it almost one-to-one} if it is one-to-one on a dense set of configurations.  

Let $\Omega$ be a (non-empty) $G$-subshift with $\omega\in\Omega$. The {\em $G$-orbit of $\omega$} is defined as the set of configurations $\{\sigma^g(\omega)\,:\,g\in G\}\subseteq\Omega$, whereas the {\em stabilizer of $\omega$} is given as the subgroup
\[
\stab_G(\omega):=\{g\in G\,:\,\sigma^g(\omega)=\omega\}\leq G
\]
of all group elements fixing $\omega$. We say $\omega$ is {\em weakly periodic} if $\abs{\stab_G(\omega)}=\infty$, whereas $\omega$ is called {\em strongly periodic} if its $G$-orbit is finite or equivalently if its stabilizer is of finite index in $G$, \ie $\abs{G:\stab_G(\omega)}<\infty$. Accordingly the (non-empty) subshift $(\Omega,G)$ is called {\em weakly aperiodic} if none of its configurations are strongly periodic, that is $\abs{G:\stab_G(\omega)}=\infty$ for all $\omega\in\Omega$, while it is called {\em strongly aperiodic} if all its configurations $\omega\in\Omega$ have trivial stabilizer $\stab_G(\omega)=\{1_G\}$, thus excluding even any weakly periodic behavior. A finitely generated group $G$ is said to {\em admit a strongly aperiodic SFT} if there exists a (non-empty) strongly aperiodic $G$-SFT.

A $G$-subshift not containing any non-empty, proper, closed and shift-invariant subset is referred to as {\em minimal}. Equivalently, every single $G$-orbit in a minimal shift $(\Omega,G)$ is dense. In the case of $\abs{\Omega}=\infty$ minimality excludes the existence of strongly periodic (but not necessarily weakly periodic) configurations.

To finish this section we recall a tool from combinatorial group theory. Given a finitely generated group $G$, its {\em (left) Cayley graph} with respect to a finite generating set $\sg\subseteq G\,\setminus\{1_G\}$ is the connected, locally finite, $\sg$-labeled digraph $\Gamma_{G,\sg}=(V,E)$ whose set of vertices is $V=G$ and whose set of (directed) edges $E\subseteq V\times V$ consists of all pairs of the form $(g,sg)$ for $g\in G$ and $s\in \sg$. An edge $e=(g,sg)$ starts at vertex $\ig(e)=g$, terminates at vertex $\tg(e)=sg$ and carries the generator $s$ as its label. Note that there are exactly $\abs{\sg}$ edges starting respectively ending at each given vertex $g\in G$, making $\Gamma_{G,\sg}$ an $\abs{\sg}$-regular digraph. We point out that the prodiscrete topology on $\ag^G$ mentioned above is induced by the usual product-space metric obtained from using the word metric on the $\sg$-Cayley graph (for an arbitrary (finite) set of generators $\sg$), and that two configurations $\omega,\omega'\in\ag^G$ are considered close if they agree on all group elements inside a large ball -- measured by the word metric -- centered at the identity element $1_G$.

\section{Robinson tilings of the plane}\label{s:Robinson}

In this section we provide a brief review of the Robinson tilings introduced in \cite{R71}, while referring the reader to \cite{JK97,R99,RR04} and the references therein for any details omitted in our exposition.

Following \cite[\S2 and \S3]{R71}, Robinson tilings are constructed using an alphabet $\agR$ of 56 square tiles obtained from rotating the 14 decorated tiles shown in Figure~\ref{f:tiles} by multiples of 90 degrees. Following common nomenclature, we call (rotations of) the leftmost tile in both rows of Figure~\ref{f:tiles} a {\em cross}, while we collectively refer to (rotations of) the 12 non-cross tiles as {\em arms}. In a valid tiling copies of those tiles are placed edge to edge filling the plane subject to the following {\em Robinson rules}:
\begin{enumerate}[label=(R\arabic*)]
\item Across every edge, shared by a pair of adjacent tiles, arrows of each type (\arcol and black) have to match head to tail.\label{rule1}
\item The parity check digits (0, 1 or 2) on opposite sides of an edge segment, shared by a pair of adjacent tiles, have to sum to 2.\label{rule2}
\end{enumerate}

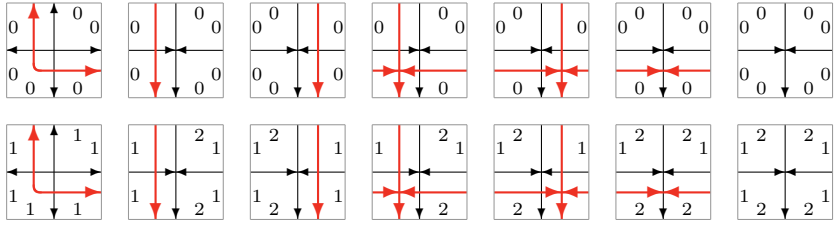
\begin{figure}[hbt]
\setlength{\unitlength}{9mm}
\begin{picture}(12.4,4)
\thinlines
\color{black}
\multiput(1.9,1.2)(1.8,0){6}{\vector(1,0){0.7}}
\multiput(1.9,3)(1.8,0){6}{\vector(1,0){0.7}}
\multiput(3.3,1.2)(1.8,0){6}{\vector(-1,0){0.7}}
\multiput(3.3,3)(1.8,0){6}{\vector(-1,0){0.7}}
\multiput(2.6,1.9)(1.8,0){6}{\vector(0,-1){1.4}}
\multiput(2.6,3.7)(1.8,0){6}{\vector(0,-1){1.4}}
\multiput(0.8,1.2)(0,1.8){2}{\vector(1,0){0.7}}
\multiput(0.8,1.2)(0,1.8){2}{\vector(-1,0){0.7}}
\multiput(0.8,1.2)(0,1.8){2}{\vector(0,1){0.7}}
\multiput(0.8,1.2)(0,1.8){2}{\vector(0,-1){0.7}}
\thicklines
\color{cmm}
\multiput(0.5,1.0)(0,1.8){2}{\vector(0,1){0.9}}
\multiput(0.6,0.9)(0,1.8){2}{\vector(1,0){0.9}}
\multiput(0.5,0.9)(0,1.8){2}{\qbezier(0,0.1)(0,0)(0.1,0)}
\multiput(2.3,1.9)(0,1.8){2}{\vector(0,-1){1.4}}
\multiput(4.7,1.9)(0,1.8){2}{\vector(0,-1){1.4}}
\multiput(5.9,1.9)(0,1.8){2}{\vector(0,-1){1.4}}
\multiput(8.3,1.9)(0,1.8){2}{\vector(0,-1){1.4}}
\multiput(5.5,0.9)(0,1.8){2}{\vector(1,0){0.4}}
\multiput(6.9,0.9)(0,1.8){2}{\vector(-1,0){1}}
\multiput(7.3,0.9)(0,1.8){2}{\vector(1,0){1}}
\multiput(8.7,0.9)(0,1.8){2}{\vector(-1,0){0.4}}
\multiput(9.1,0.9)(0,1.8){2}{\vector(1,0){0.7}}
\multiput(10.5,0.9)(0,1.8){2}{\vector(-1,0){0.7}}
\thicklines
\color{black}
\multiput(0.2,3.35)(1.8,0){7}{\makebox(0,0){\mbox{\scriptsize$0$}}}
\multiput(1.4,3.35)(1.8,0){7}{\makebox(0,0){\mbox{\scriptsize$0$}}}
\multiput(0.2,1.55)(1.8,0){7}{\makebox(0,0){\mbox{\scriptsize$1$}}}
\multiput(1.4,1.55)(1.8,0){7}{\makebox(0,0){\mbox{\scriptsize$1$}}}
\put(1.15,0.65){\makebox(0,0){\mbox{\scriptsize$1$}}}
\put(1.15,1.75){\makebox(0,0){\mbox{\scriptsize$1$}}}
\put(1.15,2.45){\makebox(0,0){\mbox{\scriptsize$0$}}}
\put(1.15,3.55){\makebox(0,0){\mbox{\scriptsize$0$}}}
\multiput(2.95,0.65)(3.6,0){2}{\makebox(0,0){\mbox{\scriptsize$2$}}}
\multiput(4.05,0.65)(3.6,0){2}{\makebox(0,0){\mbox{\scriptsize$2$}}}
\multiput(10.15,0.65)(1.8,0){2}{\makebox(0,0){\mbox{\scriptsize$2$}}}
\multiput(2.95,1.75)(3.6,0){2}{\makebox(0,0){\mbox{\scriptsize$2$}}}
\multiput(4.05,1.75)(3.6,0){2}{\makebox(0,0){\mbox{\scriptsize$2$}}}
\multiput(10.15,1.75)(1.8,0){2}{\makebox(0,0){\mbox{\scriptsize$2$}}}
\multiput(2.95,2.45)(3.6,0){2}{\makebox(0,0){\mbox{\scriptsize$0$}}}
\multiput(4.05,2.45)(3.6,0){2}{\makebox(0,0){\mbox{\scriptsize$0$}}}
\multiput(10.15,2.45)(1.8,0){2}{\makebox(0,0){\mbox{\scriptsize$0$}}}
\multiput(2.95,3.55)(3.6,0){2}{\makebox(0,0){\mbox{\scriptsize$0$}}}
\multiput(4.05,3.55)(3.6,0){2}{\makebox(0,0){\mbox{\scriptsize$0$}}}
\multiput(10.15,3.55)(1.8,0){2}{\makebox(0,0){\mbox{\scriptsize$0$}}}
\multiput(0.2,0.85)(1.8,0){3}{\makebox(0,0){\mbox{\scriptsize$1$}}}
\multiput(0.2,2.65)(1.8,0){3}{\makebox(0,0){\mbox{\scriptsize$0$}}}
\multiput(3.2,0.85)(1.8,0){2}{\makebox(0,0){\mbox{\scriptsize$1$}}}
\multiput(3.2,2.65)(1.8,0){2}{\makebox(0,0){\mbox{\scriptsize$0$}}}
\put(0.45,0.65){\makebox(0,0){\mbox{\scriptsize$1$}}}
\put(0.45,2.45){\makebox(0,0){\mbox{\scriptsize$0$}}}
\multiput(9.45,2.45)(1.8,0){2}{\makebox(0,0){\mbox{\scriptsize$0$}}}
\multiput(9.45,3.55)(1.8,0){2}{\makebox(0,0){\mbox{\scriptsize$0$}}}
\multiput(9.45,0.65)(1.8,0){2}{\makebox(0,0){\mbox{\scriptsize$2$}}}
\multiput(9.45,1.75)(1.8,0){2}{\makebox(0,0){\mbox{\scriptsize$2$}}}
\multiput(11,2.65)(1.2,0){2}{\makebox(0,0){\mbox{\scriptsize$0$}}}
\multiput(11,0.85)(1.2,0){2}{\makebox(0,0){\mbox{\scriptsize$1$}}}
\thinlines
\color{gray}
\multiput(0.1,0.5)(1.8,0){7}{\line(1,0){1.4}}
\multiput(0.1,1.9)(1.8,0){7}{\line(1,0){1.4}}
\multiput(0.1,2.3)(1.8,0){7}{\line(1,0){1.4}}
\multiput(0.1,3.7)(1.8,0){7}{\line(1,0){1.4}}
\multiput(0.1,0.5)(1.8,0){7}{\line(0,1){1.4}}
\multiput(0.1,2.3)(1.8,0){7}{\line(0,1){1.4}}
\multiput(1.5,0.5)(1.8,0){7}{\line(0,1){1.4}}
\multiput(1.5,2.3)(1.8,0){7}{\line(0,1){1.4}}
\end{picture}
\vspace*{-2ex}
\caption{The alphabet used in the Robinson tilings (displayed tiles can still be rotated giving a total of $4\cdot 14=56$ symbols). Observe that in arm tiles the convergence point of \arcol side arrow segments is always located towards the tip of the main arrow.}\label{f:tiles}
\end{figure}

Note that both rules are local, in fact nearest neighbor, and completely symmetric with respect to the $4$ cardinal directions (rows and columns) of $\zt$. The second rule in particular is forcing parity check digits in any row or column to be either constant (digit 1) or periodically alternating (between digits 0 and 2). Moreover, since crosses -- with their outward pointing arrows along all four edges -- are not allowed to be adjacent to each other (due to the first rule), those two types of rows/columns again have to strictly alternate in each tiling. As a consequence the rules guarantee the appearance of rotations of the cross (\ie leftmost tile) in the lower row of Figure~\ref{f:tiles} in an entire coset of $2\zz\times 2\zz$, sometimes called the {\em alternating crosses constraint}. Remaining locations in $\zt$ are filled by rotations of the other 13 tiles with additional crosses (rotated copies of the leftmost symbol on the top row in Figure~\ref{f:tiles}) appearing in a highly regular fashion.

Every element in the $2\zz\times 2\zz$ coset filled by the alternating crosses constraint forms a corner of some $3\times 3$ square, bound together by a contour of \arcol arrows, whose central tile is forced to be another cross (a rotation of the leftmost tile in the upper row of Figure~\ref{f:tiles}), while the remaining $4$ tiles along the square's periphery have to be arms. Four of those $3\times 3$ square patches are then grouped together, again by a \arcol square contour -- connecting the centers of those four $3\times 3$ corner patches -- to form a $7\times 7$ square patch with another (arbitrarily oriented) single cross in its center and arms emanating from there along the $4$ cardinal directions all the way to the square's border. Figure~\ref{f:supertiles} shows (examples of) such square patches.

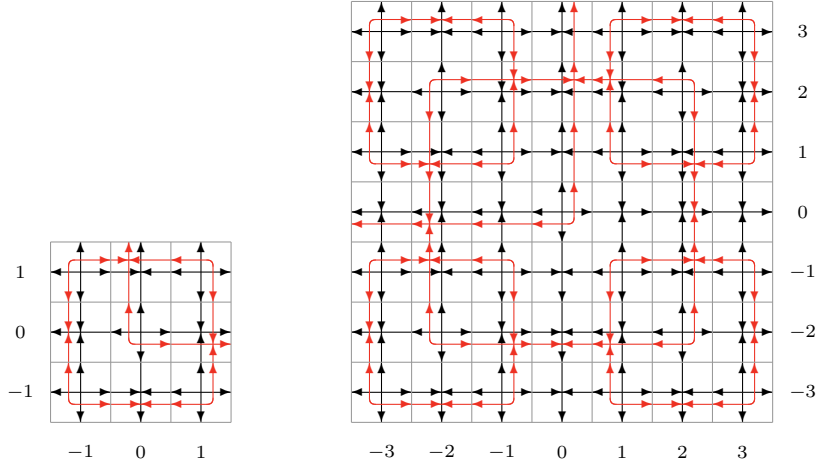
\begin{figure}[hbt]
\setlength{\unitlength}{8mm}
\begin{picture}(14,8.5)
\thinlines
\color{black}
\multiput(1.5,1.5)(2,0){2}{\vector(0,-1){0.5}}
\multiput(1.5,1.5)(0,2){2}{\vector(-1,0){0.5}}
\multiput(1.5,1.5)(2,0){2}{\vector(0,1){0.5}}
\multiput(1.5,1.5)(0,2){2}{\vector(1,0){0.5}}
\multiput(1.5,3.5)(2,0){2}{\vector(0,1){0.5}}
\multiput(3.5,1.5)(0,2){2}{\vector(1,0){0.5}}
\multiput(1.5,3.5)(2,0){2}{\vector(0,-1){0.5}}
\multiput(3.5,1.5)(0,2){2}{\vector(-1,0){0.5}}
\multiput(2,1.5)(0,2){2}{\vector(1,0){0.5}}
\multiput(3,1.5)(0,2){2}{\vector(-1,0){0.5}}
\multiput(1.5,2)(2,0){2}{\vector(0,1){0.5}}
\multiput(1.5,3)(2,0){2}{\vector(0,-1){0.5}}
\put(2,2.5){\vector(-1,0){1}}
\put(3,2.5){\vector(1,0){1}}
\put(2.5,2){\vector(0,-1){1}}
\put(2.5,3){\vector(0,1){1}}
\put(2.5,2.5){\vector(-1,0){0.5}}
\put(2.5,2.5){\vector(1,0){0.5}}
\put(2.5,2.5){\vector(0,-1){0.5}}
\put(2.5,2.5){\vector(0,1){0.5}}
\thinlines
\color{cmm}
\multiput(1.3,1.4)(2.4,0){2}{\vector(0,1){0.6}}
\multiput(1.3,3.6)(2.4,0){2}{\vector(0,-1){0.6}}
\multiput(1.4,1.3)(0,2.4){2}{\vector(1,0){0.6}}
\multiput(3.6,1.3)(0,2.4){2}{\vector(-1,0){0.6}}
\put(2,1.3){\vector(1,0){0.5}}
\put(3,1.3){\vector(-1,0){0.5}}
\put(1.3,2){\vector(0,1){0.5}}
\put(1.3,3){\vector(0,-1){0.5}}
\put(2,3.7){\vector(1,0){0.3}}
\put(3,3.7){\vector(-1,0){0.7}}
\put(3.7,2){\vector(0,1){0.3}}
\put(3.7,3){\vector(0,-1){0.7}}
\put(2.3,2.4){\vector(0,1){0.6}}
\put(2.4,2.3){\vector(1,0){0.6}}
\put(1.3,1.3){\qbezier(0,0.1)(0,0)(0.1,0)}
\put(3.7,3.7){\qbezier(0,-0.1)(0,0)(-0.1,0)}
\put(3.7,1.3){\qbezier(0,0.1)(0,0)(-0.1,0)}
\put(1.3,3.7){\qbezier(0,-0.1)(0,0)(0.1,0)}
\put(2.3,2.3){\qbezier(0,0.1)(0,0)(0.1,0)}
\put(2.3,3){\vector(0,1){1}}
\put(3,2.3){\vector(1,0){1}}

\thinlines
\color{black}
\multiput(6.5,1.5)(0,2){4}{\vector(-1,0){0.5}}
\multiput(6.5,1.5)(0,2){4}{\vector(1,0){0.5}}
\multiput(8.5,1.5)(0,2){4}{\vector(-1,0){0.5}}
\multiput(8.5,1.5)(0,2){4}{\vector(1,0){0.5}}
\multiput(10.5,1.5)(0,2){4}{\vector(-1,0){0.5}}
\multiput(10.5,1.5)(0,2){4}{\vector(1,0){0.5}}
\multiput(12.5,1.5)(0,2){4}{\vector(-1,0){0.5}}
\multiput(12.5,1.5)(0,2){4}{\vector(1,0){0.5}}
\multiput(6.5,1.5)(2,0){4}{\vector(0,-1){0.5}}
\multiput(6.5,1.5)(2,0){4}{\vector(0,1){0.5}}
\multiput(6.5,3.5)(2,0){4}{\vector(0,-1){0.5}}
\multiput(6.5,3.5)(2,0){4}{\vector(0,1){0.5}}
\multiput(6.5,5.5)(2,0){4}{\vector(0,-1){0.5}}
\multiput(6.5,5.5)(2,0){4}{\vector(0,1){0.5}}
\multiput(6.5,7.5)(2,0){4}{\vector(0,-1){0.5}}
\multiput(6.5,7.5)(2,0){4}{\vector(0,1){0.5}}
\multiput(7,2.5)(0,2){3}{\vector(-1,0){1}}
\multiput(12,2.5)(0,2){3}{\vector(1,0){1}}
\multiput(7.5,2)(2,0){3}{\vector(0,-1){1}}
\multiput(7.5,7)(2,0){3}{\vector(0,1){1}}
\multiput(8,2.5)(0,4){2}{\vector(1,0){1}}
\multiput(11,2.5)(0,4){2}{\vector(-1,0){1}}
\multiput(7.5,3)(4,0){2}{\vector(0,1){1}}
\multiput(7.5,6)(4,0){2}{\vector(0,-1){1}}
\multiput(9,4.5)(-1,0){2}{\vector(-1,0){1}}
\multiput(10,4.5)(1,0){2}{\vector(1,0){1}}
\multiput(9.5,4)(0,-1){2}{\vector(0,-1){1}}
\multiput(9.5,5)(0,1){2}{\vector(0,1){1}}
\multiput(6.5,2)(2,0){4}{\vector(0,1){0.5}}
\multiput(6.5,6)(2,0){4}{\vector(0,1){0.5}}
\multiput(6.5,3)(2,0){4}{\vector(0,-1){0.5}}
\multiput(6.5,7)(2,0){4}{\vector(0,-1){0.5}}
\multiput(7,1.5)(0,2){4}{\vector(1,0){0.5}}
\multiput(11,1.5)(0,2){4}{\vector(1,0){0.5}}
\multiput(8,1.5)(0,2){4}{\vector(-1,0){0.5}}
\multiput(12,1.5)(0,2){4}{\vector(-1,0){0.5}}
\multiput(9,1.5)(0,1){3}{\vector(1,0){0.5}}
\multiput(10,1.5)(0,1){3}{\vector(-1,0){0.5}}
\multiput(9,5.5)(0,1){3}{\vector(1,0){0.5}}
\multiput(10,5.5)(0,1){3}{\vector(-1,0){0.5}}
\multiput(6.5,4)(1,0){3}{\vector(0,1){0.5}}
\multiput(6.5,5)(1,0){3}{\vector(0,-1){0.5}}
\multiput(10.5,4)(1,0){3}{\vector(0,1){0.5}}
\multiput(10.5,5)(1,0){3}{\vector(0,-1){0.5}}
\multiput(7.5,2.5)(4,0){2}{\vector(-1,0){0.5}}
\multiput(7.5,2.5)(4,0){2}{\vector(1,0){0.5}}
\multiput(7.5,2.5)(4,0){2}{\vector(0,-1){0.5}}
\multiput(7.5,2.5)(4,0){2}{\vector(0,1){0.5}}
\multiput(7.5,6.5)(4,0){2}{\vector(-1,0){0.5}}
\multiput(7.5,6.5)(4,0){2}{\vector(1,0){0.5}}
\multiput(7.5,6.5)(4,0){2}{\vector(0,-1){0.5}}
\multiput(7.5,6.5)(4,0){2}{\vector(0,1){0.5}}
\put(9.5,4.5){\vector(-1,0){0.5}}
\put(9.5,4.5){\vector(1,0){0.5}}
\put(9.5,4.5){\vector(0,-1){0.5}}
\put(9.5,4.5){\vector(0,1){0.5}}

\thinlines
\color{cmm}
\multiput(0,0)(0,4){2}{
\multiput(0,0)(4,0){2}{
\multiput(6.3,1.4)(2.4,0){2}{\vector(0,1){0.6}}
\multiput(6.3,3.6)(2.4,0){2}{\vector(0,-1){0.6}}
\multiput(6.4,1.3)(0,2.4){2}{\vector(1,0){0.6}}
\multiput(8.6,1.3)(0,2.4){2}{\vector(-1,0){0.6}}
\put(6.3,1.3){\qbezier(0,0.1)(0,0)(0.1,0)}
\put(8.7,3.7){\qbezier(0,-0.1)(0,0)(-0.1,0)}
\put(8.7,1.3){\qbezier(0,0.1)(0,0)(-0.1,0)}
\put(6.3,3.7){\qbezier(0,-0.1)(0,0)(0.1,0)}
}
}
\multiput(7,1.3)(0,6.4){2}{\vector(1,0){0.5}}
\multiput(8,1.3)(0,6.4){2}{\vector(-1,0){0.5}}
\multiput(11,1.3)(0,6.4){2}{\vector(1,0){0.5}}
\multiput(12,1.3)(0,6.4){2}{\vector(-1,0){0.5}}
\multiput(6.3,2)(6.4,0){2}{\vector(0,1){0.5}}
\multiput(6.3,3)(6.4,0){2}{\vector(0,-1){0.5}}
\multiput(6.3,6)(6.4,0){2}{\vector(0,1){0.5}}
\multiput(6.3,7)(6.4,0){2}{\vector(0,-1){0.5}}
\multiput(7,3.7)(0,1.6){2}{\vector(1,0){0.3}}
\multiput(8,3.7)(0,1.6){2}{\vector(-1,0){0.7}}
\multiput(11,3.7)(0,1.6){2}{\vector(1,0){0.7}}
\multiput(12,3.7)(0,1.6){2}{\vector(-1,0){0.3}}
\multiput(8.7,2)(1.6,0){2}{\vector(0,1){0.3}}
\multiput(8.7,3)(1.6,0){2}{\vector(0,-1){0.7}}
\multiput(8.7,6)(1.6,0){2}{\vector(0,1){0.7}}
\multiput(8.7,7)(1.6,0){2}{\vector(0,-1){0.3}}
\multiput(7.4,2.3)(0,4.4){2}{\vector(1,0){0.6}}
\multiput(11.6,2.3)(0,4.4){2}{\vector(-1,0){0.6}}
\multiput(8,2.3)(0,4.4){2}{\vector(1,0){1}}
\multiput(11,2.3)(0,4.4){2}{\vector(-1,0){1}}
\put(9,2.3){\vector(1,0){0.5}}
\put(10,2.3){\vector(-1,0){0.5}}
\put(9,6.7){\vector(1,0){0.7}}
\put(10,6.7){\vector(-1,0){0.3}}
\multiput(7.3,2.4)(4.4,0){2}{\vector(0,1){0.6}}
\multiput(7.3,6.6)(4.4,0){2}{\vector(0,-1){0.6}}
\multiput(7.3,3)(4.4,0){2}{\vector(0,1){1}}
\multiput(7.3,6)(4.4,0){2}{\vector(0,-1){1}}
\put(11.7,4){\vector(0,1){0.5}}
\put(11.7,5){\vector(0,-1){0.5}}
\put(7.3,4){\vector(0,1){0.3}}
\put(7.3,5){\vector(0,-1){0.7}}
\put(7.3,2.3){\qbezier(0,0.1)(0,0)(0.1,0)}
\put(11.7,6.7){\qbezier(0,-0.1)(0,0)(-0.1,0)}
\put(11.7,2.3){\qbezier(0,0.1)(0,0)(-0.1,0)}
\put(7.3,6.7){\qbezier(0,-0.1)(0,0)(0.1,0)}
\put(9.7,4.4){\vector(0,1){0.6}}
\put(9.6,4.3){\vector(-1,0){0.6}}
\put(9.7,4.3){\qbezier(0,0.1)(0,0)(-0.1,0)}
\multiput(9.7,5)(0,1){3}{\vector(0,1){1}}
\multiput(9,4.3)(-1,0){3}{\vector(-1,0){1}}
\thinlines
\color{gray}
\multiput(1,1)(1,0){4}{\line(0,1){3}}
\multiput(1,1)(0,1){4}{\line(1,0){3}}
\multiput(6,1)(1,0){8}{\line(0,1){7}}
\multiput(6,1)(0,1){8}{\line(1,0){7}}
\thicklines
\color{black}
\put(1.5,0.5){\makebox(0,0){\mbox{\scriptsize$-1$}}}
\put(2.5,0.5){\makebox(0,0){\mbox{\scriptsize$0$}}}
\put(3.5,0.5){\makebox(0,0){\mbox{\scriptsize$1$}}}
\put(0.5,1.5){\makebox(0,0){\mbox{\scriptsize$-1$}}}
\put(0.5,2.5){\makebox(0,0){\mbox{\scriptsize$0$}}}
\put(0.5,3.5){\makebox(0,0){\mbox{\scriptsize$1$}}}
\put(6.5,0.5){\makebox(0,0){\mbox{\scriptsize$-3$}}}
\put(7.5,0.5){\makebox(0,0){\mbox{\scriptsize$-2$}}}
\put(8.5,0.5){\makebox(0,0){\mbox{\scriptsize$-1$}}}
\put(9.5,0.5){\makebox(0,0){\mbox{\scriptsize$0$}}}
\put(10.5,0.5){\makebox(0,0){\mbox{\scriptsize$1$}}}
\put(11.5,0.5){\makebox(0,0){\mbox{\scriptsize$2$}}}
\put(12.5,0.5){\makebox(0,0){\mbox{\scriptsize$3$}}}
\put(13.5,1.5){\makebox(0,0){\mbox{\scriptsize$-3$}}}
\put(13.5,2.5){\makebox(0,0){\mbox{\scriptsize$-2$}}}
\put(13.5,3.5){\makebox(0,0){\mbox{\scriptsize$-1$}}}
\put(13.5,4.5){\makebox(0,0){\mbox{\scriptsize$0$}}}
\put(13.5,5.5){\makebox(0,0){\mbox{\scriptsize$1$}}}
\put(13.5,6.5){\makebox(0,0){\mbox{\scriptsize$2$}}}
\put(13.5,7.5){\makebox(0,0){\mbox{\scriptsize$3$}}}
\end{picture}
\vspace*{-2ex}
\caption{Supertiles of level 1 and 2 (parity check digits suppressed). Note that (black) arrows are pushing outward along the entire border of a supertile, thus forcing all neighboring tiles to be arms. Additionally, along a supertile's periphery exactly two \arcol arrows appear in the middle of two adjacent sides reverberating the orientation of its central cross.}\label{f:supertiles}
\end{figure}

By induction on $i\in\nz$ it can be shown that inside every valid $\zt$-Robinson tiling there exist correctly tiled $(2^{i+1}-1)\times(2^{i+1}-1)$ square patches -- usually called {\em level-$i$ supertiles} -- containing cross tiles at precisely the sites in the set
\begin{equation}\label{e:ci}
\begin{split}
C_i&:=\!\bigcup_{j=0}^i\bigl(\{2^j(2k+1)-2^i:0\leq k<2^{i-j}\}\!\times\!\{2^j(2k+1)-2^i:0\leq k<2^{i-j}\}\bigr)\\
&\ \subseteq[-2^i+1,2^i-1]^2.
\end{split}
\end{equation}
Note that again every level-$i$ supertile forms part of some level-$(i+1)$ supertile, forcing a hierarchy of bigger and bigger interlinked squares.

We point out that even though the Robinson tilings are defined by rules governing pairs of adjacent tiles, these local constraints impose a rigid and global hierarchical structure. In particular, the nearest neighbor rules generate supertiles of arbitrarily high levels and at the same time force an extremely regular appearance of underlying cross tiles. The location and orientation of crosses, in turn, completely determine the mutual arrangement of supertiles. As can be seen in Figure~\ref{f:hierarchy}, showing part of a typical Robinson tiling, this forces an immense rigidity which is indeed the key ingredient in proving the strong aperiodicity of all such tilings. 

\begin{figure}[ht]
\setlength{\unitlength}{8mm}
\begin{picture}(12,10.5)
\thinlines
\color{cmm}
\multiput(1.9,2.75)(0,1){8}{\line(1,0){0.85}}
\multiput(3.75,2.75)(0,1){8}{\line(1,0){1}}
\multiput(5.75,2.75)(0,1){8}{\line(1,0){1}}
\multiput(7.75,2.75)(0,1){8}{\line(1,0){1}}
\multiput(9.75,2.75)(0,1){8}{\line(1,0){0.35}}
\multiput(2.75,2.75)(1,0){8}{\line(0,1){1}}
\multiput(2.75,4.75)(1,0){8}{\line(0,1){1}}
\multiput(2.75,6.75)(1,0){8}{\line(0,1){1}}
\multiput(2.75,8.75)(1,0){8}{\line(0,1){1}}
\thicklines
\multiput(2.25,3.25)(2,0){4}{\line(0,1){2}}
\multiput(2.25,7.25)(2,0){4}{\line(0,1){2}}
\multiput(2.25,3.25)(0,2){4}{\line(1,0){2}}
\multiput(6.25,3.25)(0,2){4}{\line(1,0){2}}
\Thicklines
\multiput(3.25,4.25)(0,4){2}{\line(1,0){4}}
\multiput(3.25,4.25)(4,0){2}{\line(0,1){4}}
\put(1.9,6.25){\line(1,0){3.35}}
\put(5.25,1.9){\line(0,1){4.35}}
\put(1.9,2.25){\line(1,0){8.2}}
\color{gray}
\thicklines
\multiput(1.9,2)(0,0.5){17}{\line(1,0){8.2}}
\multiput(2,1.9)(0.5,0){17}{\line(0,1){8.2}}
\end{picture}
\vspace*{-1.2cm}
\caption{The hierarchical structure in the Robinson tilings (black arrows and parity check digits suppressed for visibility).}\label{f:hierarchy}
\end{figure}
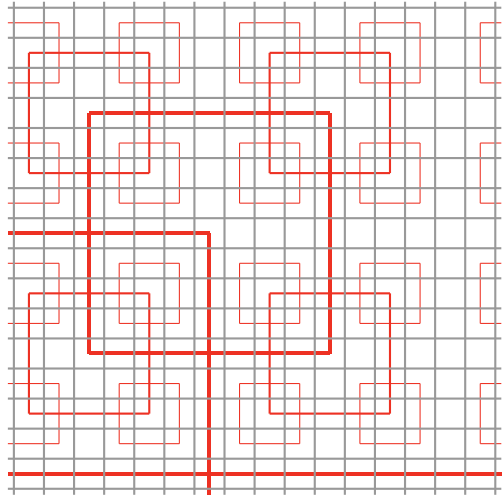

Suppose that in a given Robinson tiling the center of a level-$i$ supertiles occurs at location $(a,b)\in\zt$. We say the level-$i$ supertiles are {\em aligned in the horizontal direction} if the center crosses of all level-$i$ supertiles occur at locations with horizontal coordinates in the set $a+2^i\zz$ and they are {\em aligned in the vertical direction} if the the center crosses of all level-$i$ supertiles occur at locations with vertical coordinates in the set $b+2^i\zz$. We say the level-$i$ supertiles are {\em aligned in both directions} if they are aligned both horizontally and vertically, namely the center crosses of all level-$i$ supertiles occur in locations $(a+2^ij,b+2^ik)$, $j,k\in\zz$.

Tilings where all supertiles are aligned in both directions play an important role in our construction in Sections~\ref{s:construction} and \ref{s:minimal}. In particular, they constitute the unique minimal subshift of the Robinson tilings (see for example \cite{GJS12}) that we will use extensively to obtain the minimal factor of the Heisenberg shift of finite type to prove Theorem~\ref{t:main}.

Hereinafter we denote by $\omR\subseteq\agR^\zt$ the set of all valid $\zt$-Robinson tilings.

\section{Constructing a strongly aperiodic $\heis$-SFT}\label{s:construction}

To construct our strongly aperiodic example $(\Omega,\heis)$, we first specify its alphabet $\ag$. Next we define several sets of local rules governing how symbols can be placed on the vertices of the $\sg$-Cayley graph forming valid configurations. After this step we show that the resulting $\heis$-SFT $\Omega\subsetneq\ag^\heis$ has quite a rigid structure. We describe its $\zt$-projective subdynamics, and show that it is strongly aperiodic.

\subsection{Basic properties of the Heisenberg group}
In what follows we denote by $\heis\simeq\zt\rtimes_A\zz$ the discrete Heisenberg group, given as a semi-direct product induced by the matrix $A=\left(\begin{smallmatrix}
1& 0\\ 1 & 1 \end{smallmatrix}\right)\in\GL_2(\zz)$. $\heis$ is a 2-generator finitely presented, single-ended, non-abelian, nilpotent linear group with polynomial growth of order $4$ (for details see \cite{dlH}). Its binary operation takes the form
\[
\bigl(x,(y,z)\bigr)\cdot\bigl(a,(b,c)\bigr)=\bigl(x+a,\bigl((y,z)+A^x(b,c)\bigr)\bigr)=\bigl(x+a,(y+b,z+c+xb)\bigr)
\]
corresponding to usual matrix multiplication in $\heis$'s standard representation via upper triangular matrices, \ie where an arbitrary element $(x,(y,z))$ is seen as the unipotent integer matrix $\left(\begin{smallmatrix}1 & x & z\\ 0 & 1 & y\\ 0 & 0 & 1\end{smallmatrix}\right)$ for $x,y,z\in\zz$.

We define three particular elements:
\begin{equation*}
\xf=\left(1,(0,0)\right)\ ,\qquad
\yf=\left(0,(1,0)\right)\ ,\qquad
\zf=\left(0,(0,1)\right)
\end{equation*}
whose inverses are then given as $\xf^{-1}=\left(-1,(0,0)\right)$, $\yf^{-1}=\left(0,(-1,0)\right)$ and $\zf^{-1}=\left(0,(0,-1)\right)$. A general element $(x,(y,z))\in\heis$ can be obtained as the product $\zf^z\yf^y\xf^x=(x,(y,z))$. Hence we may fix $\sg=\{\xf,\yf,\zf\}$ as our set of generators for $\heis$. Since $\zf=\xf\yf\xf^{-1}\yf^{-1}$, the elements $\xf$ and $\yf$ are sufficient to generate $\heis$ and in fact yield the presentation
\[
\heis=\langle\xf,\yf\mid \yf^{-1}\xf\yf\xf^{-1}=\xf\yf\xf^{-1}\yf^{-1}=\yf\xf^{-1}\yf^{-1}\xf\rangle\ .
\]
For convenience we include the element $\zf$ as an additional (third) generator, resulting in the more commonly used presentation
\[
\heis=\langle\xf,\yf,\zf\mid \xf\yf\xf^{-1}\yf^{-1}=\zf,\ \xf\zf=\zf\xf,\ \yf\zf=\zf\yf\rangle\ .
\]
Note in addition that $\zf$ generates the center of $\heis$, which is isomorphic to an infinite cyclic group. The subgroups $\langle\xf,\zf\rangle\,,\langle\yf,\zf\rangle\trianglelefteq\heis$ are normal and isomorphic to $\zt$, whereas the two infinite cyclic subgroups $\langle\xf\rangle$ and $\langle\yf\rangle$ are not normal in $\heis$.

Figure~\ref{f:HCayley} shows the $\sg$-labeled left Cayley graph of $\heis$. Here -- and in all its visualizations in the remainder of this article -- we view $\langle\yf,\zf\rangle$-cosets of $\heis$ as horizontal lattices isomorphic to $\zt$, which are then connected through vertical (upward-pointing) edges representing the generator $\xf$. Due to the fact that $\heis$ is a semi-direct product induced by the matrix $A=\left(\begin{smallmatrix}
1& 0\\ 1 & 1 \end{smallmatrix}\right)$, the slope of those vertical edges depends only on the $\yf$-value of the vertices connected by them.

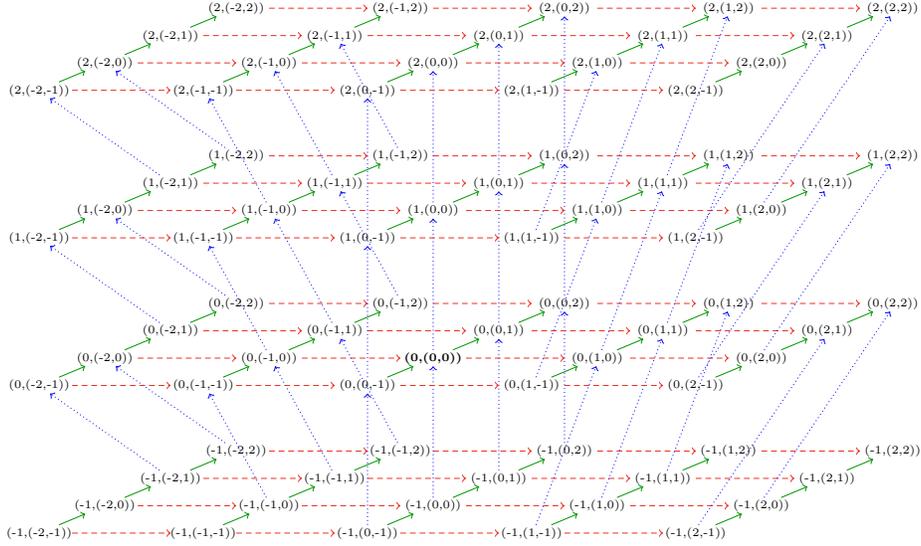
\begin{figure}[ht]
\vspace*{1ex}
\resizebox {12.5cm} {!}
{
\begin{tikzpicture}

\foreach \y in {-1,0,1}
  {
  \foreach \x in {-2,...,1}
    \foreach \z in {-1,...,2}
      \draw[line width=0.5pt,->,densely dashed,color=cmm] (3*\x+1.2*\z+0.6,2.7*\y+0.5*\z) -- (3*\x+1.2*\z+2.4,2.7*\y+0.5*\z);
  \foreach \x in {-2,...,2}
    \foreach \z in {-1,...,1}
      \draw[line width=0.5pt,->,color=green] (3*\x+1.2*\z+0.36,2.7*\y+0.5*\z+0.15) -- (3*\x+1.2*\z+0.84,2.7*\y+0.5*\z+0.35);
  \foreach \z in {-1,...,2}
    \draw[line width=0.5pt,->,dash pattern=on 0.5pt off 1.5pt,color=blue] (1.2*\z,2.7*\y+0.5*\z+0.15) -- (1.2*\z,2.7*\y+0.5*\z+2.55);
  \foreach \z in {-1,...,1}
    {
    \draw[line width=0.5pt,->,dash pattern=on 0.5pt off 1.5pt,color=blue] (1.2*\z+3.08,2.7*\y+0.5*\z+0.15) -- (1.2*\z+4.18,2.7*\y+0.5*\z+3.05);
    \draw[line width=0.5pt,->,dash pattern=on 0.5pt off 1.5pt,color=blue] (1.2*\z-1.85,2.7*\y+0.5*\z+0.65) -- (1.2*\z-2.89,2.7*\y+0.5*\z+2.55);
    }
  \foreach \z in {-1,...,0}
    {
    \draw[line width=0.5pt,->,dash pattern=on 0.5pt off 1.5pt,color=blue] (1.2*\z+6.05,2.7*\y+0.5*\z+0.15) -- (1.2*\z+8.35,2.7*\y+0.5*\z+3.55);
    \draw[line width=0.5pt,->,dash pattern=on 0.5pt off 1.5pt,color=blue] (1.2*\z-3.8,2.7*\y+0.5*\z+1.15) -- (1.2*\z-5.8,2.7*\y+0.5*\z+2.55);
    }
  }

\foreach \y in {2}
  {
  \foreach \x in {-2,...,1}
    \foreach \z in {-1,...,2}
      \draw[line width=0.5pt,->,densely dashed,color=cmm] (3*\x+1.2*\z+0.6,2.7*\y+0.5*\z) -- (3*\x+1.2*\z+2.4,2.7*\y+0.5*\z);
  \foreach \x in {-2,...,2}
    \foreach \z in {-1,...,1}
      \draw[line width=0.5pt,->,color=green] (3*\x+1.2*\z+0.36,2.7*\y+0.5*\z+0.15) -- (3*\x+1.2*\z+0.84,2.7*\y+0.5*\z+0.35);
  }

\foreach \x in {-2,...,2}
  \foreach \y in {-1,...,2}
    \foreach \z in {-1,...,2}
    {
      \ifthenelse{\x = 0 \AND \y = 0 \AND \z = 0}
        {\node (\x\y\z) at (3*\x+1.2*\z,2.7*\y+0.5*\z) {\bfseries\tiny (0,(0,0))};}
        {\node (\x\y\z) at (3*\x+1.2*\z,2.7*\y+0.5*\z) {\tiny(\y,(\x,\z))};}
    }
\end{tikzpicture}
}
\vspace*{1ex}
\caption{The $\sg$-labeled left Cayley graph of the discrete Heisenberg group where dotted, dashed, and solid edges correspond to generators $\xf,\;\yf$ and $\zf$ respectively.}\label{f:HCayley}
\end{figure}

\subsection{The first step of the construction: establishing weak aperiodicity}
Let $\agR$ be the set of 56 Robinson tiles as depicted in Figure~\ref{f:tiles}. Symbols from this alphabet will be used to fill alternating $\langle\yf,\zf\rangle$-cosets of $\heis$, to set up the same rigid hierarchical structure of crosses appearing in level-$i$ supertiles ($i\in\nz$) as seen in the original $\zt$-Robinson tilings. Remaining $\langle\yf,\zf\rangle$-cosets will then be filled with symbols from a disjoint alphabet $\agC$ actually consisting of pairs from two disjoint sets. Denote by $\agB:=\{0,1,\bfz,\bfo\}$ a binary alphabet containing two distinct versions of the digits 0 and 1 and let $\agD$ be an alphabet consisting of 4 unit square tiles as shown in Figure~\ref{f:agD}. 

\begin{figure}[ht]
\vspace*{1ex}
\begin{tikzpicture}[scale=0.5]
\foreach \x in {2,3}
 {
 \draw[line width=1.5pt] (1.25+2.4*\x,0.5) -- ++(0,1.5);
 }
\foreach \x in {1,3}
 {
 \draw[line width=1.5pt] (0.5+2.4*\x,2.0) -- ++(1.5,-1.5);
 }
\foreach \x in {0,1,2,3}
 {
 \draw[line width=1pt,color=gray] (0.5+2.4*\x,0.5) -- ++(1.5,0) -- ++(0,1.5) -- ++(-1.5,0) -- cycle;
 }
\end{tikzpicture}
\caption{The alphabet $\agD$ containing four square tiles with different combinations of up to two black line segments.}\label{f:agD}
\end{figure}
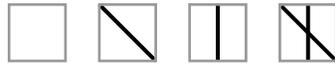

The square on the left is called a blank tile. Moreover we refer to the northwest-southeast pointing black line in the second and fourth tile as the diagonal, and the vertical line in the third and fourth tiles as the forward line segment. Forming the Cartesian product of those two alphabet sets we obtain the 16 element set $\agC:=\agB\times\agD$, called the {\em counter-alphabet}, whose symbols are thus ordered pairs $(b,s)$ with $b\in\agB$ and $s\in\agD$. One may think of the second component $s$ as being a decoration superimposed on top of the binary digit $b$ giving it the option to send additional signals throughout the configuration (propagating black lines). In the final step of our construction local constraints on these symbols will be used to implement a family of synchronized binary counters that destroy all periodicity left in the previous steps. The alphabet of our $\heis$-SFT $\Omega$ is thus the disjoint union $\ag:=\agR\sqcup\agC$ and every configuration in $\Omega$ can be seen as a function from $\heis$ to the finite alphabet set $\ag$ avoiding certain patterns.

We continue by specifying the first set of local rules. These will be nearest neighbor rules determining which symbols can be adjacent within a $\langle\yf,\zf\rangle$-coset of $\heis$. For every configuration $\omega\in\Omega$, if an element $h\in\heis$ is filled with a symbol $\omega(h)\in\agR$, then both elements $\yf h$ and $\zf h$ have to be filled with symbols from $\agR$ as well. Similarly, if $\omega(h)$ is a symbol from $\agC$ again both its neighboring symbols $\omega(\yf h)$ and $\omega(\zf h)$ have to be in $\agC$. Those two rules force each single $\langle\yf,\zf\rangle$-coset to be entirely filled with symbols of either $\agR$ or $\agC$.

Since the two infinite order elements $\yf$ and $\zf$ commute, the subgroup $\langle\yf,\zf\rangle\leq\heis$ generated by them is isomorphic to $\zt$ and hence we may identify every $\langle\yf,\zf\rangle$-coset $\langle\yf,\zf\rangle\,\xf^x=\bigl\{(x,(y,z))\in\heis\,:\,y,z\in\zz\bigr\}$ (for $x\in\zz$ fixed) with a copy of $\zt$ where the direction of $\yf$ (assumed to point to the right) generates horizontal rows and the direction of $\zf$ (pointing forward) generates horizontal columns. The next set of nearest neighbor rules affects symbols from $\agR$ as follows: Suppose element $h\in\heis$ is filled with a Robinson tile $\omega(h)\in\agR$. The symbol $\omega(\yf h)\in\agR$ seen at coordinate $\yf h$ -- still in the same $\langle\yf,\zf\rangle$-coset -- then has to be a Robinson tile which is allowed to be horizontally adjacent to $\omega(h)$ sitting on its left. Similarly the symbol $\omega(\zf h)\in\agR$ seen at coordinate $\zf h$ -- again in the same $\langle\yf,\zf\rangle$-coset -- is forced to be a Robinson tile which is allowed to sit directly above $\omega(h)$ in a valid $\zt$-Robinson tiling. Overall those rules imply that the configuration seen on any $\langle\yf,\zf\rangle$-coset filled with symbols from $\agR$ is in fact a valid $\zt$-Robinson tiling.

The subsequent set of local rules induces further constraints on $\langle\yf,\zf\rangle$-cosets filled with symbols from $\agC$. Suppose element $h\in\heis$ is filled with a symbol $\omega(h)=(b_h,s_h)\in\agC$, whose second component $s_h\in\agD$ has a diagonal segment. In this case, both symbols $\omega(\zf\yf^{-1} h)=(b_{\zf\yf^{-1} h},s_{\zf\yf^{-1} h})\in\agC$ and $\omega(\zf^{-1}\yf h)=(b_{\zf^{-1}\yf h},s_{\zf^{-1}\yf h})\in\agC$ have to continue this diagonal line, so that their second component's tile -- $s_{\zf\yf^{-1} h}$ and $s_{\zf^{-1}\yf h}$ respectively -- again has to have a black diagonal segment connecting upper-left and lower-right corners. Analogously for the forward segment, \ie $s_h\in\agD$ with a black forward segment connecting its top to its bottom edge forces the existence of forward line segments in the tiles seen in the second component of both symbols $\omega(\zf h)\in\agC$ and $\omega(\zf^{-1} h)\in\agC$. Hence the presence of a single symbol from $\agD$ with a diagonal/forward segment forces the continuation of the corresponding diagonal/forward line across all diagonally/forward adjacent symbols within the $\langle\yf,\zf\rangle$-coset. The purpose of those lines is to synchronize information in the final step of our construction and we will come back to the consequences of this {\em segments have to continue} rule later.

For now, let us continue by explaining some of the nearest neighbor constraints we need to impose along the $\xf$-direction. Since we want configurations in adjacent $\langle\yf,\zf\rangle$-cosets to alternate between Robinson tilings and certain configurations from the counter-alphabet $\agC$ we simply enforce the following. If for any $h\in\heis$, $\omega(h)\in\agR$ is a Robinson tile, then $\omega(\xf h)\in\agC$ has to be a counter-symbol, while conversely $\omega(h)\in\agC$ implies $\omega(\xf h)\in\agR$.

The final constraints in this first step of our construction are local rules forcing the {\em propagation of crosses} along direction $\xf$: Assume that the symbol $\omega(h)$ seen at $h\in\heis$ is a Robinson tile, \ie $\omega(h)\in\agR$. If it is a cross tile, the first component of the counter-symbols placed at $\xf h$ as well as at $\xf^{-1} h$ has to be one of the bold-face digits $\bfz$ or $\bfo$, whereas if the symbol $\omega(h)$ is a non-cross Robinson tile, first components in both those counter-symbols have to be the non-bold-face version of a digit $0$ or $1$. Symmetrically, seeing a counter-symbol with its first component $\bfz$ or $\bfo$ at position $h\in\heis$ forces arbitrarily oriented crosses at locations $\xf h$ and $\xf^{-1} h$, while a counter-symbol $\omega(h)=(b_h,s_h)\in\agC$ with $b_h\in\{0,1\}$ is only allowed in combination with $\omega(\xf h)$ and $\omega(\xf^{-1} h)$ both being non-cross Robinson tiles. Hence bold-face digits in the first component of counter-alphabet symbols are basically used to connect locations with crosses, while non-bold-face digits transport across the information about sites containing non-cross Robinson symbols sitting adjacent in the $\xf$-direction in neighboring $\langle\yf,\zf\rangle$-cosets.

While understanding the effects of all previous local rules is rather straight-forward, here it is definitely time to pause for a moment and check that this last set of constraints does not render our construction void by making it infeasible to fill the entire Heisenberg group with symbols without violating those rules. Recall that $\bigl(x,(y,z)\bigr)\cdot\bigl(a,(0,0)\bigr)=\bigl(x+a,(y,z+xb)\bigr)$, therefore to produce valid configurations in our $\heis$-SFT $\Omega$ we have to show the a priori not obvious fact that a sequence of valid $\zt$-Robinson tilings can be chosen for every other $\langle\yf,\zf\rangle$-coset, so that crosses do align exactly in the described fashion. This argument is the content of the following lemma.

\begin{lem}\label{l:preexistence}
There are families of regular $\zt$-Robinson tilings indexed by $\zz$, which can be used to tile alternating $\langle\yf,\zf\rangle$-cosets, allowing us to fill in the remaining sites of $\heis$ with symbols from the counter-alphabet $\agC$ and respecting all local rules specified so far.
\end{lem}

\begin{proof}
There is a $\zt$-Robinson tiling $\rho\in(\omR,\zt)$ which has its level-$i$ supertiles ($i\in\nz$ arbitrary) centered at coordinates
\[
(2^i+2^{i+1}\zz)\times(2^i+2^{i+1}\zz)\subsetneq\zt\ .
\]
The positioning of the supertiles in this tiling implies that the crosses occur exactly at the coordinates in
\[
C=\{(0,0)\}\cup\bigcup_{i\in\nz}\bigl((2^{i-1}+2^i\zz)\times(2^{i-1}+2^i\zz)\bigr)\subsetneq\zt\ .
\]
Notice that $C=\lim_{i\rightarrow\infty}C_i$, where the sets $C_i$ are as defined in \eqref{e:ci}.

For each $x\in 2\zz$, place a copy of the tiling $\rho$ into the $\langle\yf,\zf\rangle$-coset $\langle\yf,\zf\rangle\,\xf^x$, \ie let $\omega\bigl((x,(y,z))\bigr):=\rho(y,z)$. Next define a set of locations for bold-face digits
\[
B:=\{(0,0)\}\cup\bigcup_{i\in\nz}\bigl((2^{i-1}+2^i\zz)\times 2^i\zz\bigr)\subsetneq\zt
\]
and for every $x\in 2\zz+1$ extend $\omega$ to the remaining $\langle\yf,\zf\rangle$-cosets $\langle\yf,\zf\rangle\,\xf^x$ by filling in symbols from $\agC$ as follows:
\[
\omega\bigl((x,(y,z))\bigr):=\begin{cases}\bigl({\bfz},\tilb\,\bigr) \text{ iff $(y,z)\in B$}\\
\bigl(0,\tilb\,\bigr) \text{ iff $(y,z)\notin B$}\end{cases}\ .
\]
It is now straightforward to check that this gives a $\heis$-configuration $\omega\in\ag^\heis$ respecting all rules specified so far. 
Suppose $(y,z)\in C$. By definition of $C$, $y,z\in 2^{i-1}+2^i\zz$ and therefore $y+z\in 2^i\zz$, and $(y,z+y)\in B$. It is equally straightforward to check that if $(y,z)\in B$, then $(y,z+y)\in C$. Note that $\xf\cdot(x,(y,z))=(x+1,(y,z+y))$ and therefore for each $x\in2\zz$ multiplication by the generator $\xf$ on the left transforms the set of crosses into the set of bold-face digits, \ie $\xf\cdot\bigl\{(x,(y,z))\,:\,(y,z)\in C\bigr\}=\bigl\{(x+1,(y,z))\,:\,(y,z)\in B\bigr\}$ while for $x\in 2\zz+1$ the same multiplication transforms the set of bold-face digits back into the set of crosses, namely, $\xf\cdot\bigl\{(x,(y,z))\,:\,(y,z)\in B\bigr\}=\bigl\{(x+1,(y,z))\,:\,(y,z)\in C\bigr\}$.
\end{proof}

Lemma~\ref{l:preexistence} shows that the constraints defined so far do not preclude the existence of weakly periodic behavior. In fact the particular configuration $\omega$ constructed in the proof has the property that $\langle\xf^2\rangle\leq\stab_\heis(\omega)$.\\[-1ex]

\subsection{Alignment of supertiles in the $\zf$ direction}
In this subsection we show that the propagation of crosses constraint imposes some additional rigidity on the $\zt$-Robinson tilings which can be seen in our construction. Let $\omega\in\ag^\heis$ be a configuration satisfying all previous constraints. If all $\zf$-columns
\[
\langle\zf\rangle\,\yf^y\xf^x=\{(x,(y,z))\in\heis\,:\,z\in\zz\}\quad (x,y\in\zz)
\]
containing Robinson symbols see crosses spaced periodically at distance $2^i$ for a uniquely determined $i\in\nz\cup\{\infty\}$ we say that the level-$i$ supertiles in the tiling are {\em fully aligned along the $\zf$-direction}. (Here $i\in\nz$ guarantees the existence of $\overline{z}\in\zz$ such that $\{z\in\zz\,:\,\text{$\omega_{(x,(y,z))}$ is a cross tile}\}=\overline{z}+2^i\zz$, while $i=\infty$ means that $\bigl|\{z\in\zz\,:\,\text{$\omega_{(x,(y,z))}$ is a cross tile}\}\bigr|\leq 1$.) Moreover the propagation of crosses rule assures that if the supertiles are aligned along the $\zf$-direction, then the value $i$ only depends on the value of $y$ and is kept constant across an entire $\langle\xf,\zf\rangle$-coset in $\omega$. An analogous definition can be made to describe tilings that are {\em fully aligned along the $\yf$-direction}.

The next lemma shows that our nearest neighbor rules guarantee alignment in one direction.

\begin{lem}\label{l:alignment}
In each $\langle\yf,\zf\rangle$-coset filled with symbols from the Robinson alphabet $\agR$, the local rules specified so far force level-$i$ supertiles (for all $i\in\nz$) to be fully aligned along the $\zf$-direction.
\end{lem}

\begin{proof}
Let us fix a $\langle\yf,\zf\rangle$-coset, say $\langle\yf,\zf\rangle\,\xf^x\subsetneq\heis$ for some $x\in\zz$, containing symbols from $\agR$. The Robinson rules force the existence of some level-$1$ supertile, \ie a $3\times 3$ square patch bound together by a contour of \arcol arrows with crosses at its center as well as its 4 corners and arms in the remaining 4 tiles along its border. Assume this level-$1$ supertile is centered on coordinate $(x,(y,z))\in\heis$ ($y,z\in\zz$). The alternating crosses constraint then implies that the entire subcoset $\langle\yf^2,\zf^2\rangle\,\zf^{z-1}\yf^{y-1}\xf^x$ is filled with cross tiles all of which are corners of level-$1$ supertiles. Now the propagation of crosses along the $\xf$-direction guarantees the presence of cross tiles in all of $\langle\xf^2,\yf^2,\zf^2\rangle\,\zf^{z-1}\yf^{y-1}\xf^x$. Moreover it also allows us to deduce the position of the level-$1$ supertiles adjacent (along the $\zf$-direction) to the one centered on $(x,(y,z))$ as follows.

The cross tile at $(x,(y-1,z-3))$ can only be the corner of a level-$1$ supertile with center either at $(x,(y-2,z-4))$ or at $(x,(y,z-4))$. (Having its center at $(x,(y-2,z-2))$ or at $(x,(y,z-2))$ is impossible due to the ordinary $\zt$-Robinson rules, because this would force the cross tile at $(x,(y-1,z-1))$ to be part of two distinct level-$1$ supertiles.) Since the central position of a supertile is always filled with a cross, the first case would see cross tiles at both coordinates $(x,(y,z))$, the center of the level-$1$ supertile we started with, and $(x,(y-2,z-4))$. However these crosses would need to propagate along the $\xf$-direction, enforcing cross tiles at $(x-2,(y,z-2y))$ and $(x-2,(y-2,z-4-2(y-2)))=(x-2,(y-2,z-2y))$. As noted before, the entire set $\langle\xf^2,\yf^2,\zf^2\rangle\,\zf^{z-1}\yf^{y-1}\xf^x$ contains cross tiles and thus there would be crosses at $(x-2,(y-1,z-1-2y))$ and $(x-2,(y-1,z+1-2y))$ as well. Clearly this leaves no valid possibility to fill coordinate $(x-2,(y-1,z-2y))$ with a Robinson symbol. Hence the level-$1$ supertile containing $(x,(y-1,z-3))$ as one of its corners has to be centered on $(x,(y,z-4))$ and is therefore aligned along the $\zf$-direction with the one centered on $(x,(y,z))$. A symmetric argument shows that the level-$1$ supertile containing coordinate $(x,(y-1,z+3))$ has to be centered at $(x,(y,z+4))$. This local alignment then propagates across the entire $\zf$-direction, enforcing one level-$1$ supertile centered at each coordinate in $\langle\zf^4\rangle\,\zf^z\yf^y\xf^x$.

Proceeding by induction on the level of supertiles we establish the general result using a similar reasoning: A level-$(i+1)$ supertile centered on $(x,(y,z))$ contains $4$ level-$i$ supertiles centered on $(x,(y\pm 2^i,z\pm 2^i))$. Due to the induction hypothesis those give rise to two infinite families of aligned level-$i$ supertiles respectively centered on $\langle\zf^{2^{i+1}}\rangle\,\zf^{z-2^i}\yf^{y\pm 2^i}\xf^x$, which in turn have to group together to form more level-$(i+1)$ supertiles. Again there are only two possibilities for the center coordinate of the one containing the level-$i$ supertile centered on $(x,(y-2^i,z-3\cdot 2^i))$, namely $(x,(y-2^{i+1},z-2^{i+2}))$ or $(x,(y,z-2^{i+2}))$. However having cross tiles at $(x,(y-2^{i+1},z-2^{i+2}))$ and $(x,(y,z))$ implies by the propagation constraint additional crosses at
\[
\bigl(x-2,(y-2^{i+1},z-2^{i+2}-2(y-2^{i+1}))\bigr)=\bigl(x-2,(y-2^{i+1},z-2y)\bigr)
\]
and at $(x-2,(y,z-2y))$ which still is not compatible with the presence of cross tiles already fixed by the alignment of level-$i$ supertiles. Note that the family of aligned level-$i$ supertiles centered at $\langle\zf^{2^{i+1}}\rangle\,\zf^{z-2^i}\yf^{y-2^i}\xf^x$ mentioned above enforces level-$i$ supertiles centered on $\langle\zf^{2^{i+1}}\rangle\,\zf^{z-2y-2^i}\yf^{y-2^i}\xf^{x-2}$, making it impossible to fill the $\yf$-segment
\[
\bigl\{(x-2,(y-2^{i+1}+k,z-2y))\,:\,1\leq k<2^{i+1}\bigr\}
\]
with symbols from $\agR$ without violating any Robinson rules. Therefore we are once more left with the option of having aligned level-$(i+1)$ supertiles centered at $(x,(y,z))$ and $(x,(y,z-2^{i+2}))$ (and following a symmetric argument also at $(x,(y,z))$ and $(x,(y,z+2^{i+2}))$).
\end{proof}

\subsection{Pruning out weakly-periodic behavior: introducing counters}
Having established the existence of valid configurations in Lemma~\ref{l:preexistence} we now focus our attention on pruning out periodic behavior. To this end we impose a few additional local rules on sites filled with symbols from $\agC$, forcing collaboration between corresponding $\langle\yf,\zf\rangle$-cosets to form binary counters of arbitrarily large sizes which will run in a heavily synchronized way.

Fixing a value $y\in\zz$ observe the global structure of the configuration seen in the $\langle\xf,\zf\rangle$-coset
\[
\langle\xf,\zf\rangle\,\yf^y=\{(x,(y,z))\in\heis\,:\,x,z\in\zz\}\ .
\]
Since generators $\xf$ and $\zf$ commute and are of infinite order, we might again think of such a coset as an isomorphic copy of $\zt$. Forgetting about the slant in the $\xf$-direction shown in the $\sg$-Cayley graph in Figure~\ref{f:HCayley} for now we refer to $\zf$ as being horizontal, pointing to the right and $\xf$ as vertical, pointing upward. The cross-propagating $\xf$-columns seen in the restriction $\omega|_{\langle\xf,\zf\rangle\,\yf^y}$ of a configuration $\omega\in\ag^\heis$ obeying all previous rules necessarily partition $\langle\xf,\zf\rangle\,\yf^y$ into infinite vertical (slanted) strips whose uniform horizontal width $2^i$ ($i\in\nz\cup\{\infty\}$) is given by the periodic occurrences of crosses in aligned Robinson supertiles. Within each such strip we will run a binary counter whose current value, represented by symbols from the counter-alphabet $\agC$, is stored in the $\zf$-row segment confined between two successive cross-propagating $\xf$-columns and increases by $1$ when going up two layers (to the next $\agC$-filled $\zf$-row). Strips of finite width $2^i$ obviously limit the range of their binary counter, which after reaching its maximal value $2^{2^i}-1$ simply overflows, triggering a restart at value $0$. Figure~\ref{f:counters} shows part of a valid configuration on $\langle\xf,\zf\rangle\,\yf^y$ with counters of width $2^2=4$. 

\begin{figure}[ht]
\vspace*{-1ex}
\begin{tikzpicture}[scale=0.85]
\clip (0.2,1.6) rectangle + (13.5,13.5);
\foreach \bin [count=\x] in {
${\bfo}\,\tilb\ \ \;\;$ $0\,\tilb\ \ \;\;$ $0\,\tilb\ \ \;\;$ $1\,\tilb\ \ \;\;$,
${\bfo}\,\tilb\ \ \;\;$ $0\,\tilb\ \ \;\;$ $1\,\tilb\ \ \;\;$ $0\,\tilb\ \ \;\;$,
${\bfo}\,\tilb\ \ \;\;$ $0\,\tild\ \ \;\;$ $1\,\tilb\ \ \;\;$ $1\,\tild\ \ \;\;$,
${\bfo}\,\tilb\ \ \;\;$ $1\,\tilb\ \ \;\;$ $0\,\tilb\ \ \;\;$ $0\,\tilb\ \ \;\;$,
${\bfo}\,\tilb\ \ \;\;$ $1\,\tilb\ \ \;\;$ $0\,\tilb\ \ \;\;$ $1\,\tilb\ \ \;\;$,
${\bfo}\,\tilb\ \ \;\;$ $1\,\tilb\ \ \;\;$ $1\,\tilb\ \ \;\;$ $0\,\tilb\ \ \;\;$,
${\bfo}\,\tilc\ \ \;\;$ $1\,\tilc\ \ \;\;$ $1\,\tilf\ \ \;\;$ $1\,\tilc\ \ \;\;$,
${\bfz}\,\tilb\ \ \;\;$ $0\,\tilb\ \ \;\;$ $0\,\tilb\ \ \;\;$ $0\,\tilb\ \ \;\;$,
${\bfz}\,\tilb\ \ \;\;$ $0\,\tilb\ \ \;\;$ $0\,\tilb\ \ \;\;$ $1\,\tilb\ \ \;\;$,
${\bfz}\,\tilb\ \ \;\;$ $0\,\tilb\ \ \;\;$ $1\,\tilb\ \ \;\;$ $0\,\tilb\ \ \;\;$,
${\bfz}\,\tilb\ \ \;\;$ $0\,\tild\ \ \;\;$ $1\,\tilb\ \ \;\;$ $1\,\tild\ \ \;\;$,
${\bfz}\,\tilb\ \ \;\;$ $1\,\tilb\ \ \;\;$ $0\,\tilb\ \ \;\;$ $0\,\tilb\ \ \;\;$,
${\bfz}\,\tilb\ \ \;\;$ $1\,\tilb\ \ \;\;$ $0\,\tilb\ \ \;\;$ $1\,\tilb\ \ \;\;$,
${\bfz}\,\tilb\ \ \;\;$ $1\,\tilb\ \ \;\;$ $1\,\tilb\ \ \;\;$ $0\,\tilb\ \ \;\;$,
${\bfz}\,\tilb\ \ \;\;$ $1\,\tild\ \ \;\;$ $1\,\tilb\ \ \;\;$ $1\,\tild\ \ \;\;$,
${\bfo}\,\tilb\ \ \;\;$ $0\,\tilb\ \ \;\;$ $0\,\tilb\ \ \;\;$ $0\,\tilb\ \ \;\;$,
${\bfo}\,\tilb\ \ \;\;$ $0\,\tilb\ \ \;\;$ $0\,\tilb\ \ \;\;$ $1\,\tilb\ \ \;\;$,
${\bfo}\,\tilb\ \ \;\;$ $0\,\tilb\ \ \;\;$ $1\,\tilb\ \ \;\;$ $0\,\tilb\ \ \;\;$
}
{
 \foreach \o in {-5,...,10}
 {
  \node (\o\x) at (5*\o+0.5*\x,\x) {\bin};
  \draw[line width=0.5pt] (5*\o+0.5*\x-2.4,0.5+\x) -- (5*\o+0.5*\x-1.15,0.5+\x);
  \draw[line width=0.5pt] (5*\o+0.5*\x-2.075,0.35+\x) -- (5*\o+0.5*\x-1.475,0.65+\x);
 }
 \draw[line width=0.5pt,color=gray] (-1,0.35+\x) -- (20,0.35+\x);
 \draw[line width=0.5pt,color=gray] (-1,0.65+\x) -- (20,0.65+\x);
 \foreach \y in {-10,...,12}
 {
  \draw[line width=0.5pt,color=gray] (1.25*\y+0.5*\x-0.2,0.35+\x) -- (1.25*\y+0.5*\x+0.4,0.65+\x);
 }
}
\end{tikzpicture}
\vspace*{-1ex}
\caption{Binary counters of width $4$ as seen in an $\langle\xf,\zf\rangle$-coset. (A counter's most significant bit, distinguished as a bold-face digit, is always stored in a cross-propagating $\xf$-column, whereas its least significant bit is stored immediately to the left of the contiguous cross-propagating $\xf$-column to the right.) Width $4$ counters overflowing on level $6$ from the bottom, trigger diagonal segments on their most significant bits as well as forward segments across the entire counter. Width $2$ counters (on neighboring $\langle\xf,\zf\rangle$-cosets) overflowing on levels $2,\,6,\,10$ and $14$ trigger additional diagonal lines passing through the shown coset on those rows.
}\label{f:counters}
\end{figure}

Note that cross-propagating $\xf$-columns in distinct $\langle\xf,\zf\rangle$-cosets come from the central crosses of different level supertiles and are thus spaced at different periods $2^i$. Since supertiles of arbitrarily high level appear in every Robinson tiling there is no upper bound on the width of the strips seen in a valid configuration $\omega\in\Omega$ and thus also no global upper bound on the maximal counter value, \ie the number of steps it takes some counters to periodically repeat.

The actual construction of binary counters using only local rules is a bit technical, but well-known for $\zt$-SFTs or even $\zz$-cellular automata. Here we just provide a short description of the necessary constraints, while encouraging the reader to check that those rules indeed implement the counters we need. First, the presence of a cross-propagating $\xf$-column to the right should trigger addition of $1$ on the counter's least significant bit, \ie the first component of the counter-symbol sitting immediately to the left of the $\xf$-column containing crosses has to change its value in every counter step. Consequently, seeing symbols $\omega(h)=(b_h,s_h)\in\agC$ with $b_h\in\{0,1\}$ and $\omega(\zf h)=(b_{\zf h},s_{\zf h})\in\agC$ with $b_{\zf h}\in\{\bfz,\bfo\}$ around some site $h\in\heis$, symbol $\omega(\xf^2 h)=(b_{\xf^2 h},s_{\xf^2 h})\in\agC$ is only allowed to have its first component $b_{\xf^2 h}\neq b_{h}$. Note that since counter-symbol and Robinson-tile cosets alternate, the next value of the counter is stored two $\zf$-rows above, therefore compelling a vertical offset by $\xf^2$ instead of just $\xf$. Addition of bits then has to be executed locally proceeding left, across the entire counter strip, moving a possible carry along until reaching the next cross-propagating $\xf$-column (containing the counter's most significant bit). Looking at a particular site $h\in\heis$ with $\zf$-neighboring symbol $\omega(\zf h)=(b_{\zf h},s_{\zf h})\in\agC$ such that $b_{\zf h}\in\{0,1\}$, this procedure is taken care of by allowing only certain combinations of symbols from $\agB$ to appear in the first components $b_g$ of sites $g\in\{h,\,\zf h,\,\xf^2 h,\,\xf^2\zf h\}$ (also occupied by symbols from $\agC$). The list of 16 allowed locally admissible patterns is shown in Figure~\ref{f:additionpatterns}.

\begin{figure}[ht]
\begin{multline*}
\hspace*{1.5cm}
\begin{matrix}
b_{\xf^2 h}\!\! & \!b_{\xf^2\zf h}\\
b_{h}\!\! & \!b_{\zf h}
\end{matrix}
\ \in\ 
\Biggl\{\,
\begin{matrix}
0\ 0\\
0\ 0
\end{matrix}\ ,\ 
\begin{matrix}
0\ 1\\
0\ 0
\end{matrix}\ ,\ 
\begin{matrix}
1\ 0\\
1\ 0
\end{matrix}\ ,\ 
\begin{matrix}
1\ 1\\
1\ 0
\end{matrix}\ ,\ 
\begin{matrix}
0\ 1\\
0\ 1
\end{matrix}\ ,\ 
\begin{matrix}
1\ 0\\
0\ 1
\end{matrix}\ ,\ 
\begin{matrix}
0\ 0\\
1\ 1
\end{matrix}\ ,\ 
\begin{matrix}
1\ 1\\
1\ 1
\end{matrix}\ ,\ \\
\begin{matrix}
\bfz\ 0\\
\bfz\ 0
\end{matrix}\ ,\ 
\begin{matrix}
\bfz\ 1\\
\bfz\ 0
\end{matrix}\ ,\ 
\begin{matrix}
\bfo\ 0\\
\bfo\ 0
\end{matrix}\ ,\ 
\begin{matrix}
\bfo\ 1\\
\bfo\ 0
\end{matrix}\ ,\ 
\begin{matrix}
\bfz\ 1\\
\bfz\ 1
\end{matrix}\ ,\ 
\begin{matrix}
\bfo\ 0\\
\bfz\ 1
\end{matrix}\ ,\ 
\begin{matrix}
\bfz\ 0\\
\bfo\ 1
\end{matrix}\ ,\ 
\begin{matrix}
\bfo\ 1\\
\bfo\ 1
\end{matrix}
\,\Biggr\}\ .
\hspace*{1cm}
\end{multline*}
\vspace*{-2ex}
\caption{Whenever site $\zf h\in\heis$ contains a counter symbol $(b_{\zf h},s_{\zf h})\in\agC$ with $b_{\zf h}\in\{0,1\}$, we allow exactly 16 locally admissible ways to assign binary symbols to surrounding sites $h$, $\xf^2 h$ and $\xf^2\zf h$.}\label{f:additionpatterns}
\end{figure}

\begin{rem}\label{r:infcount}
For reasons that will become relevant in Section~\ref{s:minimal} we have to comment on the evolution of infinite width binary counters, which -- according to the hierarchical structure of an arbitrary Robinson tiling -- may exist in at most one $\langle\xf,\zf\rangle$-coset. Such an {\it exceptional coset} either contains a single cross-propagating $\xf$-column, dividing it into two ``halfplanes'' each of which is filled by one of two separate infinite width counters or -- in the absence of any cross-propagating $\xf$-column -- the entire $\langle\xf,\zf\rangle$-coset is occupied by a single infinite width binary counter.

In the former case the counter running in the left halfplane still has a least significant bit (located immediately to the left of the cross-propagating $\xf$-column). Hence this counter evolves as deterministically as any of its finite width cousins. The absence of a least significant bit for the counter running in the other halfplane, as well as for the unique infinite width counter in the latter case, however, allows for a choice of whether and on which (unique) $\zf$-row the counter value increases. This causes those counters to follow a rather different, partially non-deterministic behavior explained in the following lemma.
\end{rem}

In what follows we refer to an evolution from $\bfo11\ldots1$ to $\bfz00\ldots0$ as a {\em total overflow} and an evolution from $\bfz w011\ldots1$ to $\bfz w100\ldots0$ (with $w\in\{0,1\}^*$) as a {\em partial overflow}.

\begin{lem}\label{l:infcount}
An infinite width binary counter with a least significant bit increases its value in each step -- exactly as a finite width counter -- leading to at most one total overflow.

The value of an infinite width binary counter without a least significant bit stays unchanged across its entire $\langle\xf,\zf\rangle$-coset, unless there is a single increase produced by a partial or total overflow at a unique but arbitrary (\ie non-deterministic) position in its evolution.
\end{lem}

\begin{proof}
The first statement is a direct consequence of the local rule forcing a binary counter's least significant bit to periodically change between $0$ and $1$, which then triggers the usual counting including periodic partial overflows on arbitrarily long finite suffixes.

For the second part we have to analyze the effect of a couple of local patterns possibly seen in $\omega\in\Omega$ on the infinite counter's coset $\langle\xf,\zf\rangle\,\yf^y$ ($y\in\zz$): Suppose there is an $\agC$-filled site $h\in\langle\xf,\zf\rangle\,\yf^y$ with $\omega(h)_1\in\{0,{\bfz}\}$ but $\omega(\xf^2 h)_1\in\{1,{\bfo}\}$, signalling an increasing digit at $h$. The locally admissible patterns shown in Figure~\ref{f:additionpatterns} then force $\omega(\zf^{-k}h)=\omega(\zf^{-k}\xf^2 h)$ (until $\zf^{1-k} h$ hits the infinite counter's possibly existing most significant bit) as well as $\omega(\zf^kh)=1$ and $\omega(\zf^k\xf^2h)=0$ for all $k\in\nz$. As a secondary effect the same local rules then imply $\omega(\zf^k\xf^{2+2l}h)=\omega(\zf^k\xf^2h)$ and $\omega(\zf^k\xf^{-2l}h)=\omega(\zf^kh)$ for all $l\in\nz$ and all $k\in\zz$ such that site $\zf^kh$ is part of the counter. Hence the counter is stuck on the same value on all $\zf$-rows $\langle\zf\rangle\,\xf^{2-2l}h$ below, experiences a single (infinite) partial overflow between $\zf$-rows $\langle\zf\rangle\,h$ and $\langle\zf\rangle\,\xf^2h$, before staying on its newly reached value for all $\zf$-rows $\langle\zf\rangle\,\xf^{2l}h$ above.

Let us assume next that the above does not happen. The counter might still contain a local pattern with $\omega(h)_1=1$ and $\omega(\xf^2h)_1=0$, \ie a decreasing bit, at some of its $\agC$-filled sites $h\in\langle\xf,\zf\rangle\,\yf^y$. This however is only possible if the counter experiences a total overflow, forcing all its digits in $\zf$-rows $\langle\zf\rangle\,\xf^{2-2l}h$ to be $1$ and all its digits in $\zf$-rows $\langle\zf\rangle\,\xf^{2l}h$ to be $0$ (for $l\in\nz$). Therefore the counter is at its highest value (all $1$ digits) in all $\zf$-rows below and stays on the value zero for all $\zf$-rows above this unique total overflow.

In the absence of both such local patterns we have $\omega(h)_1=\omega(\xf^2h)_1$ on all counter sites $h\in\langle\xf,\zf\rangle\,\yf^y$, \ie all digits stay unchanged forever, trivially concluding the argument.
\end{proof}

Once having set up those binary counters, the final task in our construction is to synchronize all counters having the same finite width $2^i$ ($i\in\nz$). This is done by controlling the $\langle\yf,\zf\rangle$-cosets in which those counters may overflow using superimposed diagonal lines to communicate such an event across a Robinson tiling's entire net of level-$i$ supertiles. After implementing all previous parts, there is a surprisingly simple solution to this, involving one last local rule: Suppose site $h\in\heis$ is filled with a counter-symbol $\omega(h)=(b_h,s_h)\in\agC$ having $b_h\in\{\bfz,\bfo\}$, thus marking a cross-propagating $\xf$-column. Consider its next-to-nearest neighbor $\omega(\xf^2 h)=(b_{\xf^2 h},s_{\xf^2 h})\in\agC$, which then necessarily has $b_{\xf^2 h}\in\{\bfz,\bfo\}$ as well, to see whether or not the counter experiences a total overflow. If $b_h=\bfo$ and $b_{\xf^2 h}=\bfz$ this is the case (the most significant bit decreases) and we require $s_h=\tilc$ to be the $\agD$-tile containing both the diagonal as well as the forward segment, whereas in the remaining cases -- either the combination $b_h=b_{\xf^2 h}=\bfo$ or $b_h=\bfz$ -- the symbol $s_h\in\agD$ is forced to be the blank tile not containing any line segments.

We point out that while this final rule does not impose any further restrictions on the binary digits $\{0,1\}$, it actually excludes five of the 8 $\agC$-symbols having a bold-face binary digit as its first component. Additionally we can remove from our initially defined alphabet the two $\agC$-symbols combining a $0$ in its first component with a tile including a forward line segment in its second component. This is due to the fact that whenever a counter is about to undergo a total overflow none of its binary digits is 0. With these observations the cardinality of the alphabet actually used in configurations of $(\Omega,\heis)$ decreases to $56+3+6=65$.\\[-1ex]

\subsection{Alignment of supertiles in the $\yf$ direction and counter synchronization}
We prove further rigidity results about supertiles and binary counters, effectively characterizing the $\zt$-projective subdynamics seen in our construction.

The next lemma shows that diagonal segments generated during a counter overflow enforce the width of all counters whose most significant bits are located on the same northwest-southeast pointing diagonal inside a $\langle\yf,\zf\rangle$-coset to coincide.

\begin{lem}\label{l:overflow}
Let $l\in\zz\,\setminus\{0\}$. Whenever a valid configuration has two sites $(x,(y,z)),$ $(x,(y-l,z+l))\in\heis$ both containing the $\agC$-symbol $(\bfo,\tilc)$, the corresponding $\zf$-rows $\langle\zf\rangle\,\yf^y\xf^x$ and $\langle\zf\rangle\,\yf^{y-l}\xf^x$ contain binary counters of equal width.
\end{lem}

\begin{proof}
Suppose the $\zf$-row $\langle\zf\rangle\,\yf^y\xf^x$ contains bold-face digits with finite period $2^i$ ($i\in\nz$) smaller than the (possibly infinite) width of the binary counters implemented along $\langle\zf\rangle\,\yf^{y-l}\xf^x$. Propagation of the line segments emanating from the $\agC$-symbol $(\bfo,\tilc)$ seen at site $(x,(y,z))$ would force all counters along $\langle\zf\rangle\,\yf^y\xf^x$ as well as the counter having its most significant bit at site $(x,(y-l,z+l))$ to overflow on layer $\langle\yf,\zf\rangle\,\xf^x$. Counters of width $2^i$ would then repeat overflowing $2\cdot 2^{2^i}$ cosets above, where cross-propagation has displaced all bold-face digits to sites
\[
\langle\zf^{2^i}\rangle\,\zf^{z+2\cdot 2^{2^i}\cdot y}\yf^y\xf^{x+2\cdot 2^{2^i}}=\langle\zf^{2^i}\rangle\,\zf^z\yf^y\xf^{x+2\cdot 2^{2^i}}\ .
\]
Similarly, the most significant bit located at $(x,(y-l,z+l))$ propagates along the $\xf$-direction to site $h:=\bigl(x+2\cdot 2^{2^i},\bigl(y-l,z+l+2\cdot 2^{2^i}(y-l)\bigr)\bigr)$ thus also containing a bold-face digit. Since one of the newly overflowing binary counters of width $2^i$ has its most significant bit stored at position
\[
\bigl(x+2\cdot 2^{2^i},\bigl(y,z+2\cdot 2^{2^i}(y-l)\bigr)\bigr)\in \langle\zf^{2^i}\rangle\,\zf^z\yf^y\xf^{x+2\cdot 2^{2^i}}\ ,
\]
the diagonal line necessarily emanating from there would extend to the most significant bit at $h$. However, due to its larger width the corresponding binary counter would not yet experience a total overflow on layer $\langle\yf,\zf\rangle\,\xf^{x+2\cdot 2^{2^i}}$, making it impossible to have a bold-face digit with a diagonal line segment there and contradicting our initial assumption.
\end{proof}

In order to prove synchronization of our binary counters we show another rigidity result about Robinson tilings seen in $\langle\yf,\zf\rangle$-cosets of configurations in $(\Omega,\heis)$:

\begin{lem}\label{l:alignment2}
In each $\langle\yf,\zf\rangle$-coset filled with symbols from the Robinson alphabet $\agR$, the additional local rules about binary counters force level-$i$ supertiles (for all $i\in\nz$) to be also fully aligned along the $\yf$-direction.
\end{lem}

\begin{proof}
Let us start by showing that for every $i\in\nz$ (mis-)alignment of level-$i$ supertiles along $\yf$-rows is preserved across all $\agR$-filled $\langle\yf,\zf\rangle$-cosets. This follows directly from the propagation of crosses constraint: Due to Lemma~\ref{l:alignment}, the existence of a single level-$i$ supertile centered at coordinate $(x,(y,z))\in\heis$ forces the presence of level-$i$ supertiles centered on all sites $\langle\zf^{2^{i+1}}\rangle\,\zf^z\yf^y\xf^x$. Suppose that adjacent to this filled-in $\zf$-row strip of aligned supertiles there is another level-$i$ supertile centered at $(x,(y+2^{i+1},z+o))\in\heis$, which -- invoking Lemma~\ref{l:alignment} a second time -- guarantees the appearance of another infinite $\zf$-row strip of aligned level-$i$ supertiles centered on $\langle\zf^{2^{i+1}}\rangle\,\zf^{z+o}\yf^{y+2^{i+1}}\xf^x$. Here $o\in 2^{i+1}\zz$ corresponds to mutual alignment, whereas $o\notin 2^{i+1}\zz$ indicates a partial offset between these two $\zf$-row strips of supertiles. Moving up two $\langle\yf,\zf\rangle$-cosets, propagation of crosses then displaces the centers of those level-$i$ supertiles from sites
\[
\langle\zf^{2^{i+1}}\rangle\,\zf^z\yf^y\xf^x\,\cup\,\langle\zf^{2^{i+1}}\rangle\,\zf^{z+o}\yf^{y+2^{i+1}}\xf^x\subsetneq \langle\yf,\zf\rangle\,\xf^x
\]
to sites
\[
\langle\zf^{2^{i+1}}\rangle\,\zf^{z+2y}\yf^y\xf^{x+2}\,\cup\,\langle\zf^{2^{i+1}}\rangle\,\zf^{z+2y+2^{i+2}+o}\yf^{y+2^{i+1}}\xf^{x+2}\subsetneq \langle\yf,\zf\rangle\,\xf^{x+2}\ .
\]
Simplifying the $\zf$-component $\langle\zf^{2^{i+1}}\rangle\,\zf^{z+2y+2^{i+2}+o}=\langle\zf^{2^{i+1}}\rangle\,\zf^{z+2y+o}$ in the second set of coordinates establishes the preservation of the offset $o$ across consecutive $\agR$-filled $\langle\yf,\zf\rangle$-cosets as claimed.

Next we prove that the offset $o$ indeed has to be a multiple of $2^{i+1}$, \ie that arbitrary adjacent $\zf$-row strips of level-$i$ supertiles in a fixed $\langle\yf,\zf\rangle$-coset are in fact aligned along the $\yf$-direction. Presume as before that there is an entire $\zf$-row strip of aligned level-$i$ supertiles centered along $\langle\zf\rangle\,\yf^y\xf^x$. Each of those supertiles obviously constitutes one quadrant of some level-$(i+1)$ supertile, all of which -- following from Lemma~\ref{l:alignment} -- again have to align properly along the $\zf$-direction. Assuming without loss of generality that this infinite family of level-$(i+1)$ supertiles extends to the left having their central crosses placed at sites $\langle\zf^{2^{i+2}}\rangle\,\zf^{z-2^i}\yf^{y-2^i}\xf^x$, alignment between the two families of constituent level-$i$ supertiles centered at $\langle\zf^{2^{i+1}}\rangle\,\zf^z\yf^y\xf^x$ and $\langle\zf^{2^{i+1}}\rangle\,\zf^z\yf^{y-2^{i+1}}\xf^x$ is then forced by the ordinary $\zt$-Robinson rules. To show alignment on the opposite side, suppose there exists a level-$i$ supertile centered at coordinate $(x,(y+2^{i+1},z+o))\in\heis$, thus giving rise to an entire family of level-$i$ supertiles with their central crosses spaced $2^{i+1}$-periodically at sites $\langle\zf^{2^{i+1}}\rangle\,\zf^{z+o}\yf^{y+2^{i+1}}\xf^x$. Looking one $\langle\yf,\zf\rangle$-coset above, the $\zf$-row $\langle\zf\rangle\,\yf^{y+2^{i+1}}\xf^{x+1}$ is thus filled with segments of binary counters of width $2^{i+1}$ having their most significant bit stored at $\langle\zf^{2^{i+1}}\rangle\,\zf^{z+y+o}\yf^{y+2^{i+1}}\xf^{x+1}$. Using the result about the persistence of the offset $o$ across successive $\langle\yf,\zf\rangle$-cosets obtained at the beginning, we may assume without loss of generality that the binary counter whose most significant bit is located at
\[
\Bigl(x+1,\Bigl(y+2^{i+1},z+y+o-2^{i+1}\cdot\ceil{\frac{o}{2^{i+1}}}\Bigr)\Bigr)
\]
experiences a total overflow. The $\agC$-symbol seen at this location is thus a bold-face digit $\bfo$ decorated with both line segments. The forward segment then stretches across $\langle\zf\rangle\,\yf^{y+2^{i+1}}\xf^{x+1}$ making all other binary counters along its way overflow as well, whereas the diagonal segment extends along all sites
\[
\Bigl\{\Bigl(x+1,\Bigl(y+2^{i+1}-l,z+y+o-2^{i+1}\cdot\ceil{\frac{o}{2^{i+1}}}+l\Bigr)\Bigr)\,:\,l\in\zz\Bigr\}\subsetneq\heis\ .
\]
Recall from Section~\ref{s:Robinson} that the level-$(i+1)$ supertiles centered at $\langle\zf^{2^{i+2}}\rangle\,\zf^{z-2^i}\yf^{y-2^i}\xf^x$ have their crosses placed at
\[
\bigl\{x\bigr\}\times\bigl\{y-3\cdot2^i+2^j+2^{j+1}k\,:\,0\leq k<2^{i+1-j}\bigr\}\times\bigl(z-3\cdot2^i+2^j+2^{j+1}\zz\bigr)
\]
for all $j\in\{0,1,\ldots,i+1\}$.
The propagation of crosses constraint then forces bold-face binary digits at sites
\[
\bigl\{x+1\bigr\}\times\bigl\{y-3\cdot2^i+2^j+2^{j+1}k\,:\,0\leq k<2^{i+1-j}\bigr\}\times\bigl(z+y-3\cdot 2^{i+1}+2^{j+1}\zz\bigr)
\]
for all $j\in\{0,1,\ldots,i+1\}$. Those sites notably comprise the contiguous $\yf$-segment
\[
\bigl\{x+1\bigr\}\times\bigl\{y-k\,:\,0\leq k< 2^{i+1}\bigr\}\times\bigl\{z+y+2^{i+1}\bigr\}
\]
which is hit by the diagonal line forced by the binary counter undergoing a total overflow. The site of intersection
\[
\bigl(x+1,\bigl(y+o-2^{i+1}\cdot\ceil{\frac{o}{2^{i+1}}},z+y+2^{i+1}\bigr)\bigr)
\]
then has to be filled with another bold-face digit $\bfo$ decorated with both line segments. This implies that the corresponding binary counters along
\[
\langle\zf\rangle\,\yf^{y+o-2^{i+1}\cdot\ceil{\frac{o}{2^{i+1}}}}\xf^{x+1}
\]
are overflowing as well and according to Lemma~\ref{l:overflow} have to have the same width $2^{i+1}$ as the ones along $\langle\zf\rangle\,\yf^{y+2^{i+1}}\xf^{x+1}$. Hence $y+o-2^{i+1}\cdot\ceil{\frac{o}{2^{i+1}}}=y$, the unique coordinate of a $\zf$-row having the correct period of crosses, and we conclude $o\in 2^{i+1}\zz$.
\end{proof}

\begin{rem}\label{r:aligned}
Together Lemmata~\ref{l:alignment} and \ref{l:alignment2} exclude any misalignment between supertiles. The only Robinson tilings appearing in the $\agR$-filled $\langle\yf,\zf\rangle$-cosets in our $\heis$-SFT $\Omega$ thus have their level-$i$ supertiles repeated fully periodically along both the $\yf$- and the $\zf$-direction for each $i\in\nz$. This means that whenever a site $(x,(y,z))\in\heis$ contains the central cross of some level-$i$ supertile the centers of all (other) level-$i$ supertiles are located at
\[
\bigcup_{k\in\zz}\langle\yf^{2^{i+1}},\zf^{2^{i+1}}\rangle\,\zf^{z+2k\cdot y}\yf^{y}\xf^{x+2k}\ .
\]
\end{rem}

 \begin{rem}\label{r:minimal}
 We can now describe the Robinson tilings that appear in the $\mathbb Z^2$ projective subdynamics of this SFT restricted to the subgroup $\langle y,z\rangle$.  
Any pattern seen in a fully aligned Robinson tiling, as described in Remark~\ref{r:aligned}, is already a subpattern of some level-$i$ supertile for sufficiently large $i\in\nz$. Hence it occurs inside every Robinson tiling and the family of tilings seen in $\langle\yf,\zf\rangle$-cosets, as discussed in Section~\ref{s:Robinson},  comprises exactly the unique minimal subsystem contained inside the $\zt$-Robinson SFT. 
\end{rem}

\begin{cor}\label{c:sync}
Given a valid configuration in $(\Omega,\heis)$, all binary counters having the same (finite) width undergo a total overflow on the same $\langle\yf,\zf\rangle$-cosets.
\end{cor}

\begin{proof}
Note that for $i\in\nz$ the binary counters of width $2^i$ have their most significant bit stored along the $\xf$-columns propagating the central crosses of level-$(i-1)$ supertiles. (Here we explicitly include the elements in the $2\zz\times2\zz$ coset filled by the alternating crosses constraint as (centers of) level-$0$ supertiles.) According to Remark~\ref{r:aligned} those crosses appear exactly at $\bigcup_{k\in\zz}\langle\yf^{2^i},\zf^{2^i}\rangle\,\zf^{z+2k\cdot y}\yf^{y}\xf^{x+2k}$ for some $(x,(y,z))\in\heis$. Consequently the most significant bits of width $2^i$ counters in a fixed but arbitrary $\agC$-filled coset $\langle\yf,\zf\rangle\,\xf^{x+2k+1}$ ($k\in\zz$) are stored at sites $\langle\yf^{2^i},\zf^{2^i}\rangle\,\zf^{z+(2k+1)\cdot y}\yf^{y}\xf^{x+2k+1}$ forming a doubly-periodic lattice. The total overflow of one of those binary counters -- indicated by the presence of a bold-face digit $\bfo$ accompanied by the $\agD$-tile $\tilc$ with both line segments -- triggers the total overflow of all binary counters along its own $\zf$-row (propagation of the forward segment) as well as along its $\langle\yf^{-1}\zf\rangle$-diagonal (propagation of the diagonal segment) hitting and thus enforcing a total overflow of all binary counters along the remaining $\zf$-rows (propagation of corresponding forward segments). This collective overflow then repeats periodically on every $2\cdot 2^{2^i}$\!th $\langle\yf,\zf\rangle$-coset above and below. 
\end{proof}

\subsection{A strongly aperiodic shift of finite type}
The following two propositions let us conclude that the example we have constructed is indeed a strongly aperiodic shift of finite type on the discrete Heisenberg group.

\begin{prop}\label{p:existence}
The $\heis$-SFT $\Omega$ constructed above is non-empty.
\end{prop}

\begin{proof}
Non-emptiness of $\Omega$ basically follows from Lemma~\ref{l:preexistence} plus the fact that counter-cosets can be filled in without violating any of the later rules. For an explicit configuration $\omega\in\Omega$ take the particular $\zt$-Robinson tiling $\rho\in(\omR,\zt)$ used in that lemma and as before let $\omega\bigl((x,(y,z))\bigr):=\rho(y,z)$ for all $x\in 2\zz$.

The crosses appearing in $\rho$ and thus in $\omega$ still enforce the presence of bold-face $\agB$-digits precisely at all sites $(x,(y,z))\in\heis$ with
\[
x\in 2\zz+1\ \text{ and }\ (y,z)\in B=\bigl\{(0,0)\bigr\}\cup\bigcup_{i\in\nz}\bigl((2^{i-1}+2^i\zz)\times 2^i\zz\bigr)\ .
\]
Note that this set comprises the entire $\yf$-row $\langle\yf\rangle\,\xf^{-1}$ and define
\[
\omega\bigl((-1,(y,z))\bigr):=\begin{cases}\bigl({\bfo},\tilc\,\bigr) \text{ iff $(y,z)\in B$}\\
\bigl(1,\tilc\,\bigr) \text{ iff $(y,z)\notin B$}\end{cases}
\]
(the forward and diagonal lines emanating from $\langle\yf\rangle\,\xf^{-1}$ cover the entire $\langle\yf,\zf\rangle$-coset). Hence $\omega$ has all its binary counters -- independent of their width -- overflowing simultanously on layer $\langle\yf,\zf\rangle\,\xf^{-1}$, thus resulting in
\[
\omega\bigl((1,(y,z))\bigr)=\begin{cases}\bigl({\bfz},\tilb\,\bigr) \text{ iff $(y,z)\in B$}\\
\bigl(0,\tilb\,\bigr) \text{ iff $(y,z)\notin B$}\end{cases}
\]
(all blank tiles forced as no counters overflow). Deterministic evolution of the finite as well as infinite binary counters on the other $\agC$-cosets then dictates all remaining symbols: Given $k\in\zz$ and $i\in\nz$, the value of all counters of width $2^i$ on layer $\langle\yf,\zf\rangle\,\xf^{2k+1}$ is $k\!\mod 2^{2^i}$. (The value of an infinite width counter on layer $\langle\yf,\zf\rangle\,\xf^{2k+1}$ is $k$, applying the usual convention of representing negative numbers in base $2$ by an infinite prefix of $1$s.) For $k+1\equiv 0\!\mod 2^{2^i}$, all counters of width $2^j$ with $j\leq i$ overflow, triggering the occurrence of $\agD$-tiles with a forward segment at all sites $(2k+1,(y,z))\in\heis$ with $(y,z)\in \bigcup_{j\leq i}\bigl((2^{j-1}+2^j\zz)\times\zz\bigr)$ and a diagonal segment at all sites with $(y,z)\in \bigcup_{j\leq i}\{(2^{j-1}+2^j\cdot m-l,l)\,:\,l,m\in\zz\}$, hence resulting in the presence of $\agD$-tiles with both line segments at all coordinates in the intersection, \ie with $(y,z)\in \bigcup_{j\leq i}\bigl((2^{j-1}+2^j\zz)\times 2^j\zz\bigr)$ and blank tiles at all remaining coordinates (\ie in the complement of the union of both sets).
\end{proof}

\begin{prop}\label{p:aperiodicity}
The $\heis$-SFT $\Omega$ constructed above is strongly aperiodic.
\end{prop}

\begin{proof}
Let $\omega\in\Omega$ be an arbitrary configuration and let $a,b,c\in\zz$ be integers for which $p:=(a,(b,c))\in\stab_\heis(\omega)$. The periodicity $\sigma^p(\omega)=\omega$ then implies that whenever $\omega$ sees a particular symbol at site $h\in\heis$, the same symbol has to appear at site $h\cdot p$. Select $i\in\nz$ sufficiently large to satisfy $2^i>\max\{\abs{a},\abs{b},\abs{c}\}$ and take $x,y,z\in\zz$ such that $\omega\bigl((x,(y,z))\bigr)\in\ag$ sees the central cross of some level-$i$ supertile. Lemma~\ref{l:alignment} guarantees the existence of infinitely many aligned level-$i$ supertiles centered on all sites $\langle\zf^{2^{i+1}}\rangle\,\zf^z\yf^y\xf^x$ thus forcing crosses on the $\zf$-row $\langle\zf\rangle\,\yf^y\xf^x$ to be spaced periodically with distance $2^{i+1}$. Note that all other $\zf$-rows $\langle\zf\rangle\,\yf^{y+k}\xf^x$ with $-2^i<k<2^i$, $k\neq0$ forming part of this family of supertiles contain crosses at a strictly smaller period. Applying a shift by the stabilizer element $p$ forces crosses along $\langle\zf\rangle\,\yf^y\xf^x\cdot p=\langle\zf\rangle\,\yf^{y+b}\xf^{x+a}$ to occur with the same period, but at sites $\langle\zf^{2^{i+1}}\rangle\,\zf^{z+c+bx}\yf^{y+b}\xf^{x+a}$. Now propagation of crosses along the $\xf$-direction allows us to deduce what happens along $\zf$-row $\langle\zf\rangle\,\yf^{y+b}\xf^x$: there are crosses exactly at sites
\[
\langle\zf^{2^{i+1}}\rangle\,\zf^{z+c+bx-a(y+b)}\yf^{y+b}\xf^x\ .
\]
Due to the bound $\abs{b}<2^i$ this $\zf$-row is still included in the strip occupied by the infinite family of aligned level-$i$ supertiles forced by the central cross at $(x,(y,z))$, all of whose $\zf$-rows except $\langle\zf\rangle\,\yf^y\xf^x$ contain crosses at smaller periods. Hence we conclude $b=0$, which locates both $\zf$-rows $\langle\zf\rangle\,\yf^y\xf^x$ and $\langle\zf\rangle\,\yf^y\xf^x\cdot p$ within the same $\langle\xf,\zf\rangle$-coset. Consequently the value of the synchronized binary counters of width $2^{i+1}$ seen along $\langle\zf\rangle\,\yf^y\xf^{x+1}$ would have to coincide with the value seen in the width $2^{i+1}$ binary counters along $\langle\zf\rangle\,\yf^y\xf^{x+1}\cdot p=\langle\zf\rangle\,\yf^y\xf^{x+a+1}$. Since $\abs{a}<2^i\ll 2^{2^{i+1}}$ (the counter's period) this only happens for $a=0$, which leaves us with the occurrence of shifted crosses at sites
\[
\langle\zf^{2^{i+1}}\rangle\,\zf^z\yf^y\xf^x\cdot p=\langle\zf^{2^{i+1}}\rangle\,\zf^{z+c}\yf^y\xf^x\ .
\]
Comparing those to the position of the unshifted crosses at $\langle\zf^{2^{i+1}}\rangle\,\zf^z\yf^y\xf^x$, the bound $\abs{c}<2^i$ implies $c=0$ as the only option. Hence $\stab_\heis(\omega)=\bigl\{(0,(0,0))\bigr\}$ is trivial and since $\omega$ was arbitrary $(\Omega,\heis)$ is indeed strongly aperiodic.
\end{proof}

\section{Overflow coordination -- constructing the minimal, strongly aperiodic factor}\label{s:minimal}

Since we explicitly control both alphabet and local rules of our strongly aperiodic $\heis$-SFT, we might wonder, whether our construction already forces minimality of $\Omega$. While this is not the case, in this section we will modify $\Omega$ to obtain a new strongly aperiodic SFT that is an at most two-to-one extension of a minimal strongly aperiodic subshift. 
 
Minimality fails for the SFT $\Omega$ for two distinct reasons, one of which is not too hard to circumvent.
Thus far we do not control the offsets between $\langle\yf,\zf\rangle$-cosets on which binary counters of different width experience a total overflow. For example our SFT $\Omega$ contains two configurations $\omega,\,\omega'\in\Omega$, both having counter symbols on all sites in $\langle\xf^2,\yf,\zf\rangle\,\xf$, but such that in $\omega$ counters of all widths overflow simultaneously on the coset $\langle\yf,\zf\rangle\,\xf^{-1}$ -- this is the case for the particular configuration constructed in Proposition~\ref{p:existence} -- whereas in $\omega'$ only counters of width at least $2^2=4$ overflow on $\langle\yf,\zf\rangle\,\xf^{-1}$, while all counters of width $2^1=2$ periodically overflow on the cosets $\langle\yf,\zf\rangle\,\xf^{1+8\zz}$. It is not hard to see that the (disjoint) $\heis$-orbits of configurations $\omega$ and $\omega'$ are separated by a positive distance\footnote{Patterns of support $\bigl\{(x,(y,z))\,:\,x\in\{0,1\}\,\wedge\,y,z\in\{0,1,\ldots,7\}\bigr\}\subsetneq\heis$ occuring somewhere in $\omega$ are disjoint from patterns of the same support appearing in $\omega'$.}. Hence neither of them is dense in our present construction, contradicting minimality of $\Omega$.

In order to fix this issue of distinct-width overflow non-coordination we slightly enlarge our alphabet $\agC$ and specify a couple of additional local rules forcing overflows on larger width counters to induce simultaneous overflows on counters of all smaller widths. To do so, let us introduce a single new square tile $\tila$, replacing some of the occurrences of $\tilc$, while keeping the tile $\tilc$ unchanged in other places.

Hence from here on we consider
\[
\agDtil:=\bigl\{\tilb,\tild,\tilf,\tilc,\tila\bigr\}\ ,
\]
forming $\agCtil:=\agB\times\agDtil$ just as before. A first supplementary constraint nevertheless forces the new {\em overflow-coordination tile} $\tila$ to be admissible as second component in a symbol of $\agCtil$ only if the $\agB$-digit in its first component is a bold-face $\bfo$. Thus effectively adding only one symbol to $\agC$, the total of occurring symbols (a proper subset of $\agCtil$) increases from $9$ to $10$ fixing the cardinality of the alphabet actually used in our modified $\heis$-SFT at $66$.

Next construct $\omT\subsetneq\bigl(\agR\sqcup\agCtil\bigr)^\heis$ by literally applying all constraints imposed on our previously defined $\heis$-SFT $\Omega\subsetneq\bigl(\agR\sqcup\agC\bigr)^\heis$, except for adding the possibility of having a symbol $\bigl({\bfo},\tila\,\bigr)$ whenever we had a symbol $\bigl({\bfo},\tilc\,\bigr)$ before. This increased flexibility, apparently veering us even further away from a minimal system, is then reduced by invoking the following pair of local rules. For every $\omega\in\omT$ and every site $h\in\heis$ with $\omega(h)\in\agCtil$ we enforce:
\begin{enumerate}[label=(OC\arabic*)]
\item If either $\omega(h)_1=0$ or $\omega(h)_1=\omega(\xf^2h)_1=1$, \ie there is no overflow on the corresponding binary counter's digit, having $\omega(\yf h)=\bigl({\bfo},\tilc\,\bigr)$ forces $\omega(\yf^{-1}h)=\bigl({\bfo},\tilc\,\bigr)$ and symmetrically having $\omega(\yf^{-1}h)=\bigl({\bfo},\tilc\,\bigr)$ forces $\omega(\yf h)=\bigl({\bfo},\tilc\,\bigr)$.\label{frule1}
\item If either $\omega(h)_1=1$ and $\omega(\xf^2h)_1=0$, \ie there is an overflow on the corresponding binary counter's digit, or $\omega(h)=\bigl({\bfo},\tila\,\bigr)$, \ie the entire counter overflows, the presence of a bold-face binary digit in the $\agCtil$-symbol at site $\yf h$ forces $\omega(\yf h)=\bigl({\bfo},\tila\,\bigr)$, while the presence of a bold-face binary digit in the $\agCtil$-symbol at $\yf^{-1} h$ forces $\omega(\yf^{-1} h)=\bigl({\bfo},\tila\,\bigr)$.\label{frule2}
\end{enumerate}

The global effect of those local constraints is that each overflow occurring on certain digits of a binary counter with a larger width triggers a total overflow of each smaller width counter whose most significant bit is aligned along the $\yf$-direction with one of those overflowing digits. The horizontal line segment seen in the newly introduced square tile $\tila$ visually propagates the necessary information about such an event uninterruptedly across larger distances in the $\yf$-direction. Figure~\ref{f:counterovercoord} illustrates this {\em counter overflow coordination} in two distinct $\langle\yf,\zf\rangle$-coset patterns occurring in configurations in $\omT$.

\begin{figure}[htb]
\vspace*{1ex}
\begin{tikzpicture}[scale=0.9]
\clip (-0.2,0.4) rectangle + (13.8,9.65);

\fill[color={black!10}] (0.7,10-0.8) rectangle (1.1,10-5.7);
\fill[color={black!10}] (0.7,10-8.8) rectangle (1.1,10-10.2);
\fill[color={black!10}] (8.7,10-0.8) rectangle (9.1,10-3.2);
\fill[color={black!10}] (8.7,10-8.8) rectangle (9.1,10-10.2);
\fill[color={black!10}] (10.5,10-0.8) rectangle (10.9,10-3.2);
\fill[color={black!10}] (10.5,10-4.8) rectangle (10.9,10-7.2);
\fill[color={black!10}] (10.5,10-8.8) rectangle (10.9,10-10.2);

\foreach \z in {0,8}
{
 \foreach \x in {1,3,5,7,9,11}
 {
  \foreach \y in {3,5,7,11,13,15,19}
  {
   \node at (\z+0.45*\x,10-0.5*\y) {$\tila$};
  }
 }
 \foreach \x in {0,4,8,12}
 {
  \foreach \y in {5,13}
  {
   \node at (\z+0.45*\x,10-0.5*\y) {$\tila$};
  }
 }
}

\foreach \x in {0,1,3,4,5,7,8,9,11,12}
{
 \node at (0.45*\x,10-0.5*1) {$\tilc$};
 \node at (0.45*\x,10-0.5*9) {$\tila$};
 \node at (0.45*\x,10-0.5*17) {$\tilc$};
 \node at (8+0.45*\x,10-0.5*1) {$\tilc$};
 \node at (8+0.45*\x,10-0.5*9) {$\tilc$};
 \node at (8+0.45*\x,10-0.5*17) {$\tilc$};
}

\node at (0.45*2,10-0.5*1) {$\tilb$};
\node at (0.45*2,10-0.5*17) {$\tilb$};
\node at (0.45*6,10-0.5*1) {$\tilc$};
\node at (0.45*6,10-0.5*9) {$\tila$};
\node at (0.45*6,10-0.5*17) {$\tilc$};

\node at (8+0.45*2,10-0.5*1) {$\tilb$};
\node at (8+0.45*2,10-0.5*17) {$\tilb$};
\node at (8+0.45*6,10-0.5*1) {$\tilb$};
\node at (8+0.45*6,10-0.5*9) {$\tilb$};
\node at (8+0.45*6,10-0.5*17) {$\tilb$};

\foreach \z in {0,8}
{
 \foreach \x in {1,3,5,7,9,11}
 {
  \foreach \y in {1,3,5,7,9,11,13,15,17,19}
  {
   \node[color=cmm] (\z\x\y) at (\z+0.45*\x,10-0.5*\y) {$\bfo$};
  }
  \foreach \y in {0,2,4,6,8,10,12,14,16,18,20}
  {
   \node[color=cmm] (\z\x\y) at (\z+0.45*\x,10-0.5*\y) {$1$};
  }
 }
 \foreach \x in {0,4,8,12}
 {
  \foreach \y in {1,5,9,13,17}
  {
   \node[color=cmm] (\z\x\y) at (\z+0.45*\x,10-0.5*\y) {$\bfo$};
  }
  \foreach \y in {0,2,3,4,6,7,8,10,11,12,14,15,16,18,19,20}
  {
   \node[color=cmm] (\z\x\y) at (\z+0.45*\x,10-0.5*\y) {$1$};
  }
 }
 \foreach \bin [count=\y] in {
 $1$,$0$,$1$,$0$,$0$,$0$,$0$,$0$,$1$,$1$,$0$,$0$,$1$,$0$,$0$,$0$,$0$,$0$,$0$,$1$,$0$
 }
 {
  \node[color=cmm] (\z10\y) at (\z+0.45*10,10-0.5*\y+0.5) {\bin};
 }
}

\foreach \y in {1,9,17}
{
 \node[color=cmm] (6\y) at (0.45*6,10-0.5*\y) {$\bfo$};
}
\foreach \y in {0,2,3,4,5,6,7,8,10,11,12,13,14,15,16,18,19,20}
{
 \node[color=cmm] (6\y) at (0.45*6,10-0.5*\y) {$1$};
}

\foreach \bin [count=\y] in {
$0$,$\bfz$,$1$,$1$,$1$,$1$,$1$,$1$,$1$,$1$,$1$,$1$,$0$,$0$,$0$,$1$,$0$,$\bfz$,$1$,$1$,$1$
}
{
 \node[color=cmm] at (0.45*2,10-0.5*\y+0.5) {\bin};
}

\foreach \bin [count=\y] in {
$0$,$\bfz$,$1$,$1$,$1$,$1$,$1$,$0$,$0$,$0$,$0$,$0$,$1$,$0$,$0$,$1$,$0$,$\bfz$,$1$,$1$,$1$
}
{
 \node[color=cmm] at (8+0.45*2,10-0.5*\y+0.5) {\bin};
}

\foreach \bin [count=\y] in {
$0$,$\bfz$,$1$,$1$,$1$,$1$,$1$,$0$,$0$,$\bfz$,$1$,$1$,$1$,$1$,$1$,$0$,$0$,$\bfz$,$1$,$1$,$1$
}
{
 \node[color=cmm] at (8+0.45*6,10-0.5*\y+0.5) {\bin};
}

\node at (6.7,6.4) {\large $\Longrightarrow$};
\node at (6.7,5.8) {\large $32$};
\node at (6.7,5.2) {\large counter};
\node at (6.7,4.6) {\large steps};
\node at (6.7,4) {\large $\Longrightarrow$};
\end{tikzpicture}
\caption{Partial overflows -- shown as lightgray shading -- trigger the new overflow-coordination tile $\tila$. (For better visibility $\agDtil$-symbols are only drawn on most significant bits.) The pattern on the right appears $2\cdot 32=64$ $\langle\yf,\zf\rangle$-cosets above the one on the left.}\label{f:counterovercoord}
\end{figure}
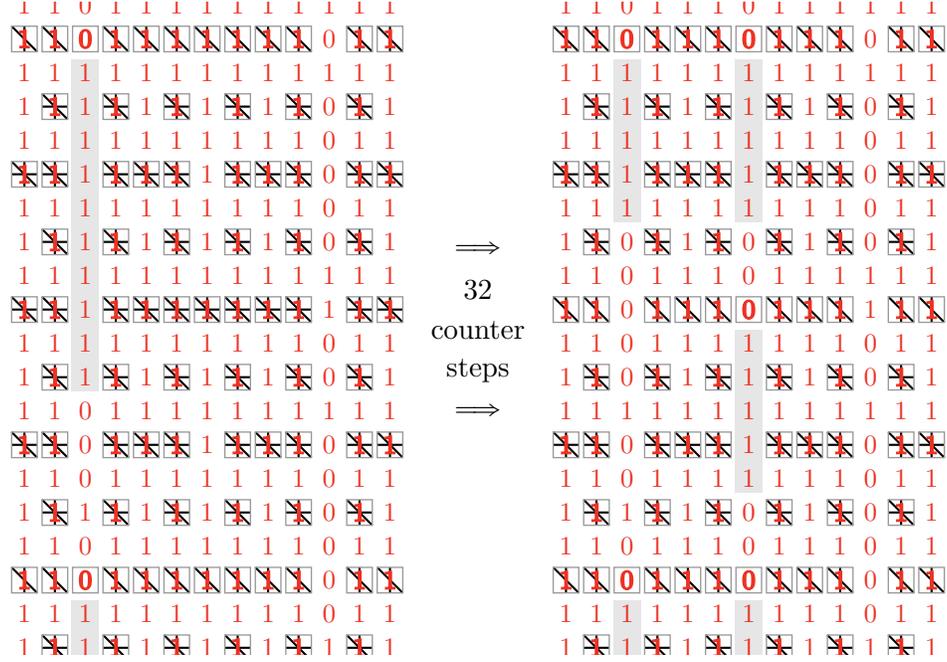

The first consequence of the new rules is the following lemma.

\begin{lem}\label{l:yrow}
In a valid configuration in $\bigl(\omT,\heis\bigr)$ no $\yf$-row can contain both $\agDtil$-tiles $\tilc$ and $\tila$.
\end{lem}

\begin{proof}
Since both tiles $\tilc$ and $\tila$ are always paired with a bold-face $\agB$-digit $\bfo$, the resulting $\agCtil$-symbols only appear on $\langle\yf,\zf\rangle$-cosets where at least some binary counters -- those of width $2$ -- experience a total overflow. Due to the complete alignment of Robinson supertiles (see Remark~\ref{r:aligned}) the bold-face binary digits seen in an arbitrary $\agCtil$-filled $\yf$-row $\langle\yf\rangle\,h$ ($h\in\heis$) group together in contiguous $\yf$-segments of length $2^i-1$ for some fixed $i\in\nz\cup\{\infty\}$. Those $\yf$-segments are inevitably separated by a width $1$ gap containing a non bold-face digit. (Here the infinite length case allows for either two half-row segments separated by a width $1$ gap or for an entire $\yf$-row filled with bold-face digits.) The second part of Rule~\ref{frule2} excludes the coexistence of symbols $\bigl({\bfo},\tilc\,\bigr)$ and $\bigl({\bfo},\tila\,\bigr)$ within a single such $\yf$-segment, whereas Rule~\ref{frule1} forbids adjacent segments to change type.
\end{proof}

\begin{lem}\label{l:overcoord}
Valid configurations in $\bigl(\omT,\heis\bigr)$ obey overflow coordination: for every $i\in\nz$, a total overflow of width $2^i$ counters on some $\langle\yf,\zf\rangle$-coset forces all width $2^j$ counters with $j<i$ to also undergo a total overflow on this particular coset. Moreover, the occurrence of a (partial, respectively, total) overflow of the unique infinite width counter not having a least significant bit forces a simultaneous total overflow of all finite width counters as well as a total overflow of the possibly existing second infinite width counter\footnote{As discussed in Remark~\ref{r:infcount}, an infinite width counter having a least significant bit might run in the left half of the unique exceptional $\langle\xf,\zf\rangle$-coset.} on the same $\langle\yf,\zf\rangle$-coset.
\end{lem}

\begin{proof}
Choose $h\in\heis$ such that $\langle\xf,\zf\rangle\,h$ is filled with width $2^i$ counters with their most significant bits stored in the $\xf$-columns of $\langle\xf,\zf^{2^i}\rangle\,h$ and which periodically experience total overflows on all $\zf$-rows in $\langle\xf^{2^{2^i+1}},\zf\rangle\,h$. The alignment of Robinson supertiles -- given by Lemma~\ref{l:alignment} and Lemma~\ref{l:alignment2} -- forces all sites in the contiguous $\yf$-segment $S=\bigl\{\yf^k\zf^{2^{i-1}} h\,:\,k\in\{1,2,\ldots,2^{i-1}-1\}\bigr\}$ to contain most significant bits of binary counters of smaller width. Thus those positions are filled with $\agCtil$-symbols having a bold-face digit as their first component. The overflowing counter of width $2^i$ sees a binary digit $1$ at $\zf^{2^{i-1}} h$ and a binary digit $0$ at $\xf^2\zf^{2^{i-1}} h$. According to Rule~\ref{frule2} this then triggers $\agCtil$-symbols at all sites in $S$ to be $\bigl({\bfo},\tila\,\bigr)$ which in turn forces all binary counters running in $\langle\xf,\zf\rangle$-cosets $\langle\xf,\zf\rangle\,\yf^k h$ with $1\leq k\leq 2^{i-1}-1$ to experience a simultaneous total overflow on $\zf$-row $\langle\zf\rangle\,\yf^k h$. Note that at least one of those cosets is filled with a width $2^j$ counter for every $j\in\{1,2,\ldots,i-1\}$, finishing the proof for the first part.

For the second statement, let $h\in\heis$ such that $\langle\xf,\zf\rangle\,h$ is the (unique) exceptional coset and suppose that the infinite width counter not having a least significant bit sees an overflowing digit at site $h$. This immediately implies that the digits at all positions in $\{\zf^lh\,:\,l\in\nz\}$ overflow as well, which in turn forces the spread of symbols $\bigl({\bfo},\tila\,\bigr)$ across arbitrarily long $\yf$-segments $S_i=\bigl\{\yf^k\zf^{z_i}h\,:\,k\in\{1,2,\ldots,2^i-1\}\bigr\}$ ($i,z_i\in\nz$ with $z_i\leq 2^i$) causing the simultaneous total overflow of all finite width counters on layer $\langle\yf,\zf\rangle\,h$. Recall from Lemma~\ref{l:infcount} that this situation occurs at most once in the infinite width counter's entire evolution.

Now consider the possibly existing second infinite width counter whose least significant bit on $\zf$-row $\langle\zf\rangle\,h$ we assume to be stored at position $\zf^{-z}h$ for some $z\in\nz$. Suppose this counter does not experience a total overflow on $\zf$-row $\langle\zf\rangle\,h$. Hence there certainly exists $i\in\nz$ such that site $\zf^{-z-2^i+1}h$ does not contain an overflowing digit. However, according to Lemma~\ref{l:infcount}, an overflow on this digit then has to occur in the next $2^{2^i}-1$ counter steps, \ie on some $\zf$-row $\langle\zf\rangle\,\xf^{2l}h$ with $1\leq l\leq 2^{2^i}-1$. By Rule~\ref{frule2} this overflow necessarily triggers overflow-coordination symbols $\bigl({\bfo},\tila\,\bigr)$ along an entire $\yf$-segment extending across neighboring $\langle\xf,\zf\rangle$-cosets $\langle\xf,\zf\rangle\,\yf^kh$ with $1\leq k\leq 2^i-1$. In particular this would cause a total overflow of width $2^i$ counters on layer $\langle\yf,\zf\rangle\,\xf^{2l}h$. By the previous argument we know those counters also undergo a total overflow on layer $\langle\yf,\zf\rangle\,h$. This conflicts with their values repeating exactly with periodicity $2^{2^i}$, concluding our argument.
\end{proof}

The results we have established allow us to now show that $\omT$ factors onto an aperiodic sofic subshift inside $\Omega$, via an almost one-to-one factor map.

\begin{prop}\label{p:blockcode}
There exists a $1$-block code from $\omT$ into $\Omega$ which is at most $2$-to-$1$ and whose range consists of exactly those configurations in $\Omega$ which respect the overflow coordination condition described in Lemma~\ref{l:overcoord}.
\end{prop}

\begin{proof}
Define a $1$-block code $\phi:\,\omT\rightarrow\Omega$ induced by a local map $\Phi$ acting as the identity on all symbols from $\agR$ as well as $\agC$ while sending the new symbol $\bigl({\bfo},\tila\,\bigr)$ back to its original version $\bigl({\bfo},\tilc\,\bigr)$. By definition of $\omT$, the replacements induced by $\phi$ respect all local rules of $\Omega$. Since $\phi$ does not affect counter values, configurations in $\omT$ (whose counters are overflow coordinated by Lemma~\ref{l:overcoord}) then necessarily map to configurations in $\Omega$ obeying the same overflow coordination constraint. To show that the range of $\phi$ indeed coincides with this particular subshift we construct possible preimages.

Let $\omega\in\Omega$ be a configuration whose binary counters are overflow coordinated. Note that by definition $\Phi^{-1}(\omega(h))$ is uniquely determined for all sites $h\in\heis$ except for those containing a symbol $\bigl({\bfo},\tilc\,\bigr)$. Consider an inclusion-maximal (finite or infinite) $\yf$-segment $S\subseteq\langle\yf\rangle\,h$ ($h\in\heis$) filled with $\agC$-symbols whose first components are bold-face digits. If any of the sites in $S$ sees a symbol different from $\bigl({\bfo},\tilc\,\bigr)$ no site in the preimage of $\omega|_S$ is allowed to contain the symbol $\bigl({\bfo},\tila\,\bigr)$ by Rule~\ref{frule2}. Therefore the entire preimage $\Phi^{-1}(\omega|_S)$ is again uniquely determined by definition of the local map $\Phi$. Next we assume $\omega(s)=\bigl({\bfo},\tilc\,\bigr)$ for all $s\in S$. If $S=\{\yf^kh\,:\,a<k<b\}$ with $a<b\in\zz$ is a finite $\yf$-segment of length $b-a-1$, then one of its two neighboring $\langle\xf,\zf\rangle$-cosets, $\langle\xf,\zf\rangle\,\yf^ah$, respectively $\langle\xf,\zf\rangle\,\yf^bh$, contains binary counters of width $2(b-a)$. Overflow coordination forces the value of those counters seen on $\zf$-row $\langle\zf\rangle\,\yf^ah$, respectively $\langle\zf\rangle\,\yf^bh$, to equal $2^{(b-a)}-1$ modulo $2^{2(b-a)}$. Hence their $(b-a)$th binary digit, one of which is just located at site $\yf^ah$, respectively $\yf^bh$, undergoes an overflow. Rule~\ref{frule2} then forces the preimage of $\omega|_S$ to be entirely filled with symbols $\bigl({\bfo},\tila\,\bigr)$. If $S$ is an infinite $\yf$-segment, the assumption $\omega(s)=\bigl({\bfo},\tilc\,\bigr)$ for all sites $s\in S$ implies that all finite width counters experience a total overflow on layer $\langle\yf,\zf\rangle\,h$. Since this event can occur at most once across the entire configuration $\omega$ and since each $\langle\yf,\zf\rangle$-coset can at most contain one $\yf$-row containing an infinite $\yf$-segment $S$, Lemma~\ref{l:yrow} establishes an upper bound of $2$ for the number of possible preimages of $\omega$. The two choices of filling the entire segment $S$ with $\bigl({\bfo},\tilc\,\bigr)$ or $\bigl({\bfo},\tila\,\bigr)$ in fact appear in exactly two situations; namely whenever $S=\langle\yf\rangle\,h$ is an entire $\yf$-row or if $S$ is a half-row such that the binary digit of the infinite width counter seen at the (unique) site adjacent to $S$ along the $\yf$-direction does not overflow. In all other cases the $\phi$-preimage is unique showing that $\phi$ is invertible on a dense subset of image configurations.
\end{proof}

Finally we prove that there is an almost one-to-one factor map from $\omT$ to a minimal system.

\begin{prop}\label{p:minsofic}
$\omT$ factors onto a strongly aperiodic, minimal (sofic) $\heis$-shift.
\end{prop}

\begin{proof}
Define a $1$-block code $\phi:\,\omT\rightarrow\bigl(\{A,C\}\sqcup\agB\bigr)^\heis$ induced by a local map $\Phi$ sending cross tiles to the letter $C$, arm tiles to the letter $A$, and projecting all $\agCtil=\agB\times\agDtil$ symbols onto their first component. As an immediate consequence of the propagation of crosses constraint along each $\xf$-column, letters $C$ always alternate with bold-face $\agB$-digits, whereas letters $A$ always alternate with non bold-face $\agB$-digits. Hence in $\phi(\omT)$ the knowledge of all symbols on a single $\agB$-coset is sufficient to recover all symbols on each $\langle\yf,\zf\rangle$-coset filled by the two new letters $A$ and $C$. Applying the same proof as for Proposition~\ref{p:aperiodicity} lets us conclude that $\phi(\omT)$ is still strongly aperiodic.

To prove minimality of the image system, let $\omega\in\phi(\omT)$ be an arbitrary configuration and consider its restriction $p:=\omega|_F$ to an arbitrary but finite support $F=\supp(p)\subsetneq\heis$. Without loss of generality, assume $\omega|_{\langle\yf,\zf\rangle}$ to be filled with $\agB$-symbols and choose $k\in\nz$ such that $F\subseteq\bigcup_{l=0}^{2^k}Q_k\xf^l$, where
\[
Q_k:=\bigl\{(0,(y,z)):\,\abs{y}<2^k\wedge\abs{z}<2^k\bigr\}\subsetneq\heis\ .
\]
We now have to distinguish two cases.

First suppose the base square $Q_k$ intersects only $\zf$-rows containing binary counters of finite width. Hence in all $\zf$-rows $\langle\zf\rangle\,\yf^y$ with $\abs{y}<2^k$, bold-face digits appear with periodicity some (bounded) power of $2$. Therefore the pattern $\omega|_{Q_k}$ extends to a {\em complete counter array}, \ie there is a rectangular patch
\[
R_i:=\bigl\{(0,(y,z)):\,\abs{y}<2^{i-1}\wedge\abs{z}\leq 2^{i-1}\bigr\}\subsetneq\langle\yf,\zf\rangle
\]
of shape $(2^i-1)\times(2^i+1)$ for some $i>k$ and an element $h\in\langle\yf,\zf\rangle$ such that $Q_k\subseteq R_ih$ and $\omega(sh)\in\{{\bfz}, {\bfo}\}$ for all
\[
s\in\bigl\{(0,(y,z)):\,\abs{y}<2^{i-1}\wedge\abs{z}=2^{i-1}\bigr\}\subsetneq R_i\ ,
\]
while $\omega(sh)\in\{0,1\}$ for all $s\in\bigl\{(0,(0,z)):\,\abs{z}<2^{i-1}\bigr\}\subsetneq R_i$. Overflow coordination then implies that the pattern $\omega|_{R_ih}$ is completely characterized by the value of its largest binary counter, namely a natural number between $0$ and $2^{2^i}-1$, stored in the $\zf$-segment $\bigl\{(0,(0,z):\,-2^{i-1}\leq z<2^{i-1}\bigr\}h\subsetneq R_ih$. It follows from Remark~\ref{r:minimal} that
the $\{A,C\}$-pattern $\omega|_{\xf R_ih}$ has to appear in every $\{A,C\}$-layer of every configuration in $\phi(\omT)$. Taking such an occurrence $\omega'|_{(\xf R_ih)g}=\omega|_{\xf R_ih}$ ($g\in\heis$) in an arbitrary configuration $\omega'\in\phi(\omT)$ the pattern $\omega'|_{(R_ih)g}$ is necessarily another complete counter array of shape $(2^i-1)\times(2^i+1)$ uniquely determined by some value between $0$ and $2^{2^i}-1$. Moving up through the stack of $\agB$-patterns $\bigl(\omega'|_{(R_ihg)\xf^{2l}}\bigr)_{l=0}^{2^{2^i}-1}$ this (increasing) value eventually has to match the value given by $\omega|_{R_ih}$. Hence we are sure to find a copy of $\omega|_{R_ih}$ inside $\omega'$. Deterministic evolution of finite width counters then assures to see the entire pattern $p=\omega|_F$ inside $\omega'$.

In the remaining case the base square $Q_k$ intersects a $\zf$-row, say $\langle\zf\rangle\,g$ for some $g\in Q_k$, contained in the unique exceptional $\langle\xf,\zf\rangle$-coset of $\omega$. Since $\langle\zf\rangle\,g$ sees at most one bold-face $\agB$-digit we cannot extend the support of the pattern $\omega|_{Q_k}$ to obtain a complete counter array. Instead we use that $Q_k\setminus\langle\zf\rangle\,g$ can be covered by a rectangular patch $R_{i+1}h$ of shape $(2^{i+1}-1)\times(2^{i+1}+1)$ for some $i>k$ and $h\in\langle\zf\rangle\,g$, such that each restriction $\omega|_{R_i\zf^{\pm 2^{i-1}}\yf^{\pm 2^{i-1}}h}$ to one of $R_{i+1}h$'s four corner rectangles of shape $(2^i-1)\times(2^i+1)$ is a complete counter array, while neither of the two symbols $\omega(\zf^{\pm 2^i}h)\in\agB$ is a bold-face digit. Since $\langle\zf\rangle\,h$ is part of the exceptional $\langle\xf,\zf\rangle$-coset of $\omega$, $\zf$-row $\langle\zf\rangle\,\yf^{2^j}h$ is filled with binary counters of finite width $2^{j+1}$ for each $j\in\nz_0$. Choosing $j>i$ there exists a site $h'\in\{\yf^{2^j}h,\zf^{2^j}\yf^{2^j}h\}$ for which $\omega|_{R_{i+1}h'}$ sees bold-face digits in the same locations as $\omega|_{R_{i+1}h}$. Moreover staying within the same $\langle\yf,\zf\rangle$-coset guarantees $\omega|_{(R_{i+1}\setminus\langle\zf\rangle)h'}=\omega|_{(R_{i+1}\setminus\langle\zf\rangle)h}$ (four coinciding complete counter arrays). Recall that we assumed the support $F$ of our initial pattern $p$ to be contained within the stack $\bigcup_{l=0}^{2^k}Q_k\xf^l$. Hence the choices of $j>i>k$ and $h'$ ensure a counter suffix (\ie segment of least significant digits) of no less than $2^i$ of the $2^{j+1}$ binary digits to lie outside of $(R_{i+1}\cap\langle\zf\rangle)h'$. This yields enough room to accommodate whatever behaviour -- either staying constant or showing a single increase instantly triggering a total overflow of all finite width counters as described by Lemmas~\ref{l:infcount} and~\ref{l:overcoord} -- the non-deterministic infinite width counter running in $\omega|_{\langle\xf,\zf\rangle\,h}$ might show inside $\bigcup_{l=0}^{2^k}Q_k\xf^l$. Climbing up through $\langle\yf,\zf\rangle$-layers by multiples of $2\cdot 2^{2^i}$ keeps all values of counters with widths at most $2^i$ unchanged, so that $\omega|_{(R_{i+1}\setminus\langle\zf\rangle)h'\xf^n}=\omega|_{(R_{i+1}\setminus\langle\zf\rangle)h'}$ for all $n\in 2^{2^i+1}\zz$. Each such successive jump however increases the value of the width $2^{j+1}$ counters stored in $\zf$-rows $\langle\zf\rangle\,h\xf^n$ by $2^{2^i}$ modulo $2^{2^{j+1}}$, thus producing in one periodic cycle all binary values on the counters' $2^{j+1}-2^i$ most significant bits. Choosing the correct value of $n$ we therefore reencounter a shifted version of the base pattern $\omega|_{R_{i+1}h'\xf^n}=\omega|_{R_{i+1}h}$. As this new copy is located entirely in the complement of $\omega$'s exceptional $\langle\xf,\zf\rangle$-coset, we are back in case $1$, concluding our minimality argument just as before.
\end{proof}

We are now ready to conclude the proof of our main theorem.

\begin{proof}[Proof of Theorem~\ref{t:main}]
We point out that the factor map $\phi:\,\omT\twoheadrightarrow\phi(\omT)$ constructed in Proposition~\ref{p:minsofic} is almost one-to-one, producing unique preimages $\phi^{-1}(\omega)$ on the dense subset of configurations $\omega\in\phi(\omT)$ not containing an exceptional coset. 
\end{proof}

To conclude, we finally touch upon the second, maybe less obvious reason for the non-minimality of $\Omega$ and thus also $\omT$. In fact, given any configuration $\omega\in\phi(\omT)$, the propagation of crosses constraint in itself is only strong enough to uniquely reconstruct all Robinson tiles located inside some finite-level supertile. (Similarly counter values determine the eliminated $\agDtil$-tiles.) Nevertheless there is some ambiguity in recovering Robinson symbols, namely arm tiles pointing in the $\yf$-direction, seen outside the union of all finite-level supertiles. To obtain minimality those would have to be restricted further, mimicking the implicit (non-local) rules forced on their behaviour inside finite-level supertiles by the propagation of crosses constraint. Consequently it seems rather difficult to transform $\omT$ itself into a minimal $\heis$-SFT by adding more symbol decorations and further local rules.
Hence the existence of an explicit construction of a strongly aperiodic, minimal $\heis$-SFT remains an open problem.

\end{document}